\documentclass[12pt]{amsart}
\usepackage{amsmath}	
\usepackage{amssymb}
\usepackage{caption}
\usepackage{subcaption}
\usepackage{accents}
\usepackage{tikz}
\usetikzlibrary{calc,intersections,arrows,automata}

\DeclareSymbolFont{matha}{OML}{txmi}{m}{it}
\DeclareMathSymbol{\varu}{\mathord}{matha}{117}
\DeclareMathSymbol{\varv}{\mathord}{matha}{118}
\DeclareMathSymbol{\varw}{\mathord}{matha}{119}

\newtheorem{theorem}{Theorem}[section]
\newtheorem{lemma}[theorem]{Lemma}
\newtheorem{proposition}[theorem]{Proposition}
\newtheorem{corollary}[theorem]{Corollary}

\theoremstyle{remark}
\newtheorem*{remark}{Remark}

\theoremstyle{definition}
\newtheorem{definition}[theorem]{Definition}

\numberwithin{equation}{section}

\newcommand{\ubar}[1]{\underaccent{\bar}{#1}}

\newcommand{\et}{\quad\mbox{and}\quad}
\newcommand{\bA}{\mathbb{A}}
\newcommand{\bC}{\mathbb{C}}

\newcommand{\bQ}{\mathbb{Q}}
\newcommand{\bR}{\mathbb{R}}
\newcommand{\bZ}{\mathbb{Z}}
\newcommand{\cA}{{\mathcal{A}}}
\newcommand{\cC}{{\mathcal{C}}}
\newcommand{\cI}{{\mathcal{I}}}
\newcommand{\cS}{{\mathcal{S}}}
\newcommand{\disp}{\displaystyle}
\newcommand{\dist}{\mathrm{dist}}
\newcommand{\homega}{\widehat{\omega}}

\newcommand{\image}{\mathrm{Im}}
\newcommand{\norm}[1]{\|\hspace*{2pt}#1\hspace*{1pt}\|}

\newcommand{\pv}{\varv}
\newcommand{\pw}{\varw}
\newcommand{\chibot}{{\ubar{\chi}}}
\newcommand{\chitop}{{\bar{\chi}}}

\newcommand{\psibot}{{\ubar{\psi}}}
\newcommand{\psitop}{{\bar{\psi}}}
\newcommand{\tbigwedge}{\normalsize{\bigwedge}}
\newcommand{\tA}{\widetilde{A}}
\newcommand{\talpha}{\widetilde{\alpha}}
\newcommand{\tbeta}{\widetilde{\beta}}
\newcommand{\tchi}{\widetilde{\chi}}
\newcommand{\tcC}{\widetilde{\cC}}

\newcommand{\tL}{\widetilde{L}}
\newcommand{\tP}{\widetilde{P}}
\newcommand{\ts}{\tilde{s}}
\newcommand{\tS}{\widetilde{S}}
\newcommand{\trho}{\widetilde{\rho}\,}

\newcommand{\tuL}{\widetilde{\uL}}

\newcommand{\tua}{{\widetilde\ua}}
\newcommand{\tuP}{{\widetilde\uP}}

\newcommand{\ua}{\mathbf{a}}
\newcommand{\ub}{\mathbf{b}}
\newcommand{\ue}{\mathbf{e}}
\newcommand{\uE}{\mathbf{E}}
\newcommand{\uL}{\mathbf{L}}
\newcommand{\uP}{{\mathbf{P}}}
\newcommand{\uR}{{\mathbf{R}}}
\newcommand{\uu}{\mathbf{u}}
\newcommand{\uv}{\mathbf{v}}
\newcommand{\ux}{\mathbf{x}}
\newcommand{\uX}{\mathbf{X}}
\newcommand{\uy}{\mathbf{y}}

\usepackage{adjustbox}

\newcommand{\veq}{\mathrel{\rotatebox{90}{$=$}}}
\newcommand{\vpp}{\mathrel{\rotatebox{90}{$>$}}}

\newcommand{\dpe}{\mathrel{\rotatebox{45}{$\le$}}}
\newcommand{\vpe}{\mathrel{\rotatebox{90}{$\ge$}}}

\begin{document}

\baselineskip=15.3pt  

\title[Best versus uniform Diophantine approximation]
{Best versus uniform Diophantine approximation}
\author{Martin Rivard-Cooke}
\address{
   D\'epartement de Math\'ematiques\\
   Universit\'e d'Ottawa\\
   150 Louis-Pasteur\\
   Ottawa, Ontario K1N 6N5, Canada}
\email{MartinRivardCooke@hotmail.com}
\author{Damien Roy}
\address{
   D\'epartement de Math\'ematiques\\
   Universit\'e d'Ottawa\\
   150 Louis-Pasteur\\
   Ottawa, Ontario K1N 6N5, Canada}
\email{droy@uottawa.ca}
%

\begin{abstract}
Let $0<m<n$ be integers, and let $K_\pw$ denote the completion of a number 
field $K$ at a non-trivial place $\pw$.  For each non-zero $\uu\in K_\pw^n$,
let $\omega_{m-1}(\uu)$ denote the exponent of best approximation to $\uu$
by vector subspaces of $K_\pw^n$ of dimension $m$ defined over $K$, and let
$\homega_{m-1}(\uu)$ denote the corresponding exponent of uniform 
approximation.  Finally, let $\cS_{m,n}$ denote the set of all pairs 
$(\homega_{m-1}(\uu),\omega_{m-1}(\uu))$ where $\uu$ runs through all
points of $K_\pw^n$ with linearly independent coordinates over $K$.  In this 
paper we use parametric geometry of numbers to study this spectrum $\cS_{m,n}$,
noting at first that it is independent of the choice of $K$ and $\pw$.  We 
may thus assume that $K=\bQ$ and $K_\pw=\bR$.  In this context, Schmidt 
and Summerer proposed conjectural descriptions for $\cS_{1,n}$ and $\cS_{n-1,n}$
which were confirmed by Marnat and Moshchevitin for each $n\ge 2$.  We 
give an alternative proof of their result based on the PhD thesis of 
the first author, highlighting the duality between the two spectra.   
In his thesis, the first author generalized the conjecture to any pair 
$(m,n)$ and proved it to be true also for $\cS_{2,4}$.  We present this
as well, but show that this natural conjecture fails for $\cS_{3,5}$.  
Moreover, the part 
of $\cS_{3,5}$ that we succeed to compute here suggests a complicated 
boundary for that set, possibly not semialgebraic.   We also give a 
qualitative description of $\cS_{m,n}$ for a general pair $(m,n)$. 
\end{abstract}

\maketitle

\section{Introduction}
\label{sec:intro}

In his 1967 landmark paper \cite{Sc1967}, W.~M.~Schmidt fixes a number 
field $K\subset\bR$ (resp.~$K\subset\bC$) and, given integers $d,m,n$ with $0<d<n$
and $0<m<n$, he studies approximation to $d$-dimensional vector subspaces $U$ 
of $\bR^n$ (resp.~of $\bC^n$) by $m$-dimensional vector subspaces $V$ of 
$\bR^n$ (resp.~of $\bC^n$) defined over $K$, in terms of their height $H(V)$.  
In \cite{La2009b},  M.~Laurent revisited this theory for $K=\bQ\subset \bR$ and for 
subspaces $U$ of $\bR^n$ of dimension $d=1$.  Assuming that $U$ is 
not contained in a proper subspace of $\bR^n$ defined over $\bQ$,
which is no true restriction, he defined for each $m=1,\dots,n-1$, an exponent 
$\omega_{m-1}(U)$ of approximation to $U$ by subspaces $V$ of $\bR^n$ 
of dimension $m$ defined over $\bQ$.
As $U$ is spanned by a single vector, these can be viewed as projective invariants 
$\omega_0(\uu),\dots,\omega_{n-2}(\uu)$ attached to points $\uu$ of $\bR^n$ 
with linearly independent coordinates over $\bQ$.
Then, he proved inequalities relating $\omega_{m-1}(\uu)$ and $\omega_m(\uu)$
for each integer $m$ with $1\le m\le n-2$.   These follow from Schmidt's going-up and 
going-down theorems \cite[Theorems 9 and 10]{Sc1967}, but self-contained proofs 
are also given by Y.~Bugeaud and M.~Laurent in \cite[Sections 5 and 6]{BL2010}.
In \cite{R2016}, the second author used parametric geometry of numbers as 
developed in \cite{R2015} to give an alternative proof of these inequalities 
and showed that, together with Dirichlet's
lower bound $\omega_0(\uu)\ge 1/(n-1)$, they determine 
the spectrum of $(\omega_0,\dots,\omega_{n-2})$, namely the set of all 
values that this sequence takes at points $\uu$ of $\bR^n$ with $\bQ$-linearly 
independent coordinates.

The invariants $\omega_0,\dots,\omega_{n-2}$ are often referred to as \emph{ordinary} 
exponents of approximation or as exponents of \emph{best} approximation.  They have 
natural counterparts $\homega_0,\dots,\homega_{n-2}$ called exponents of \emph{uniform} 
approximation.  For $n=3$, Laurent determined in \cite{La2009a} the spectrum of the full
sequence of exponents $(\homega_0,\dots,\homega_{n-2},\omega_0,\dots,\omega_{n-2})$,
but in higher dimension the problem is open.  Combining \cite[Proposition 3.1]{R2016} with 
parametric geometry of numbers from \cite{R2015}, this Diophantine problem can be 
translated into a combinatorial problem about a class of functions called $n$-systems
(see Section \ref{ssec:pgn:tool} below), but this still presents a challenge.

In this paper, we are interested in the spectrum of the pair $(\homega_{m-1},\omega_{m-1})$ 
for each integer $m=1,\dots,n-1$.  For $m=1$ and for the dual case $m=n-1$, these pairs of 
exponents are classical objects introduced by Khintchine and Jarn\'{\i}k, and 
their spectra were recently determined by A.~Marnat and N.~Moshchevitin in \cite{MM2020}, 
culminating work in \cite{Ja1950, Ja1954, La2009a, Mo2012, SS2013b}. 
Their result already had impact, for example, in the study of uniform rational approximation to 
real points on quadratic hypersurfaces \cite{KM2019}.
To compute these spectra, Marnat and Moshchevitin proceed in two steps.  For the case $m=1$,
they first fix a point $\uu\in\bR^n$ with $\bQ$-linearly independent coordinates and use 
classical techniques of Diophantine approximation, based on sequences of minimal points 
in $\bZ^n$ attached to $\uu$, to derive a polynomial inequality between $\homega_0(\uu)$ 
and $\omega_0(\uu)$, thus showing that the spectrum of $(\homega_0,\omega_0)$ is 
contained in a certain semialgebraic subset of $[1/(n-1),\infty]\times[1/(n-1),\infty]$.  As a 
second step, they use parametric geometry of numbers and a generalization of the \emph{regular}
$n$-systems introduced by Schmidt and Summerer in \cite{SS2013b} to show that any point 
of this subset lies in the spectrum of $(\homega_0,\omega_0)$, namely that it is of the form 
$(\homega_0(\uu),\omega_0(\uu))$ for some $\uu\in\bR^n$ with $\bQ$-linearly independent 
coordinates.  For the pair of dual exponents $(\homega_{n-2},\omega_{n-2})$, they 
proceed in a similar way.  Note that, for $m=1$, 
a quantitative form of the inequality of Marnat and Moshchevitin, 
inspired by constructions of N.~A.~V. Nguyen in \cite[Chapter 3]{Ng2014}, is
established in \cite{NPR2020} in terms of measures of approximation to points of $\bR^n$ 
by rational lines in $\bR^n$.  The argument is again based on sequences of minimal points. 

In his PhD thesis \cite{Ri2019}, the first author, M.~Rivard-Cooke, gave a short 
algebraic argument based only on 
parametric geometry of numbers from \cite{R2015} to establish the inequalities of Marnat 
and Moschevitin, generalizing what Schmidt and Summerer had done in dimension $n=4$ 
in \cite{SS2013b} based on their original version of the theory in \cite{SS2013a}.   This thus 
provides a complete analysis of the spectra of $(\homega_0,\omega_0)$ and 
$(\homega_{n-2},\omega_{n-2})$ by means of parametric geometry of numbers.  
In Section~\ref{sec:MM}, we present his argument in a way that highlights the 
duality between the two pairs of exponents.

Rivard-Cooke also extended the family of $n$-systems built by Marnat and
Moshchevitin in \cite[Section 6]{MM2020} to produce, via parametric geometry of numbers,
a large set of points in the spectrum of $(\homega_{m-1},\omega_{m-1})$ 
for any given integer $m$ with $1\le m\le n-1$, and he conjectured that this set 
is the full spectrum.  For $m=1$ and $m=n-1$, this conjecture, which we recall
in section~\ref{sec:dr}, is due to Schmidt and Summerer (see the last 
sentence of \cite[Section~3]{SS2013b}), and is shown to be true by the  
work of Marnat and Moschevitin in \cite{MM2020}.  Rivard-Cooke proved his
conjecture in dimension $n=4$ for the remaining case $m=2$.  We give the proof
in section~\ref{sec:Ri}.  It is possible however that these are the only cases where 
the conjecture holds, since it happens to be false in the next simplest case
$n=5$ and $m=3$, namely for the spectrum of $(\homega_2,\omega_2)$ in
dimension $5$.  Sections \ref{sec:dim5} and \ref{sec:more} of the paper are devoted to 
that particular spectrum.  The part of it that we succeed in computing suggests a 
relatively complicated boundary containing arcs of several 
algebraic curves.   It seems plausible that this boundary is in fact a union of 
infinitely many such arcs, and thus that the spectrum itself is not a semialgebraic 
subset of $\bR^2$ (a finite union of subsets of $\bR^2$ each of which is defined 
by finitely many polynomial inequalities).  It would be interesting to settle this 
question, as the only known examples of non-semialgebraic spectra, those
constructed in \cite[Section 5]{RR2020}, are quite artificial.

In general, all of the above extends to the situation where $\bQ$ is replaced 
by a number field $K$ and $\bR$ is replaced by the completion of $K$ at a 
place $\pw$, using the theory developed by A.~Po\"{e}ls and D.~Roy 
in \cite{PR2023}.  For an Archimedean place, we recover the setting of Schmidt in 
his paper \cite{Sc1967} mentioned above.  We recall below the relevant 
definitions and results. 
 
%
%

\section{Definitions and main results}
\label{sec:dr}

As in \cite{PR2023}, let $K$ be a finite algebraic extension of $\bQ$ and let 
$d=[K:\bQ]$ be its degree.  For each (non-trivial) place $\pv$ of $K$, we denote by $|\ |_\pv$ the 
corresponding absolute value of $K$ extending the usual absolute value of $\bQ$ 
if $\pv$ is Archimedean or its usual $p$-adic absolute value if $\pv$ lies above 
a prime number $p$.   By continuity, it extends uniquely to an absolute value 
on the completion $K_\pv$ of $K$ at $\pv$, which we also denote by $|\ |_\pv$.     
We denote the local degree of $\pv$ by $d_\pv=[K_\pv:\bQ_\pv]$, where $\bQ_\pv$ 
stands for the topological closure of $\bQ$ in $K_\pv$.  We write $\pv\mid\infty$ if $\pv$ is 
Archimedean, and $\pv\nmid\infty$ otherwise.  In the first case, $\bQ_\pv$ naturally 
identifies with $\bR$, and $K_\pv$ identifies with $\bR$ or $\bC$.


\subsection{Local norms}
Let $m$, $n$ be integers with $1\le m\le n$ and let $(\ue_1,\dots,\ue_n)$ denote
the canonical basis of $K^n$.  We set $N=\binom{n}{m}$ and denote by 
$(\uE_1,\dots,\uE_N)$ the basis of $\tbigwedge^mK^n$ over $K$ formed by the $N$
products $\ue_{i_1}\wedge\cdots\wedge\ue_{i_m}$ with $1\le i_1<\cdots<i_m\le n$, 
in some order.    Then, for any place $\pv$ of $K$, we define the norm of a point 
$\uX=\sum_{j=1}^N X_j\uE_j$ in $\tbigwedge^mK_\pv^n$ to be 
\[
 \|\uX\|_\pv 
  =\begin{cases} 
    (|X_1|_v^2+\cdots+|X_N|_v^2)^{1/2} &\text{if $\pv\mid\infty$,} \\[2pt]
    \max\{|X_1|_v,\dots,|X_N|_v\} &\text{otherwise.}
    \end{cases}
\]
When $m=1$, we have  $\tbigwedge^mK_\pv^n=K_\pv^n$, and the norm of 
a point $\ux=(x_1,\dots,x_n)$ in $K_\pv^n$ is simply its usual Euclidean norm 
if $\pv\mid\infty$, and its maximum norm otherwise.  


\subsection{Projective distance}
We fix a place $\pw$ of $K$, an integer $n\ge 1$, and a non-zero
point $\uu$ of $K_\pw^n$.   As in \cite[Section~3]{PR2023}, we define the 
\emph{projective distance} between $\uu$ and a non-zero point $\ux$ of 
$K_\pw^n$ to be
\[
 \dist(\uu,\ux) = \frac{\norm{\uu\wedge\ux}_\pw}{\norm{\uu}_\pw\norm{\ux}_\pw}.
\]
It depends only on the lines spanned by $\uu$ and $\ux$ in $K_\pw^n$.  When 
$\pw$ is Archimedean, it represents the sine of the angle between $\uu$ and $\ux$ 
for the Euclidean structure of $K_\pw^n$ attached to the norm $\norm{\ }_\pw$.
We also define the \emph{projective distance} between $\uu$ and a subspace $V$ 
of $K_\pw^n$ of dimension $m$ to be the minimal distance from $\uu$ to a non-zero 
point  of $V$.  By \cite[Lemma~3.6]{PR2023}, this is given algebraically by
\[
 \dist(\uu,V) 
  = \frac{\norm{\uu\wedge\uX}_\pw}{\norm{\uu}_\pw\norm{\uX}_\pw},
\]
where $X$ is any generator of the one-dimensional subspace $\tbigwedge^m V$ of
$\tbigwedge^mK_\pw^n$, namely where $\uX=\ux_1\wedge\cdots\wedge\ux_m$ for 
a basis $(\ux_1,\dots,\ux_m)$ of $V$ over $K_\pw$.


\subsection{Exponents of approximation}
Let $n > m \ge 1$ be integers, let $\pw$ be a place of $K$ and let 
$\uu$ be a non-zero element of $K_\pw^n$.   A subspace $V$ of $K_\pw^n$ of 
dimension $m$ is \emph{defined over $K$} if it admits a basis $(\ux_1,\dots,\ux_m)$
made of elements of $K^n$.  Then, $X=\ux_1\wedge\cdots\wedge\ux_m$
belongs to $\tbigwedge^mK^n$ and, as Schmidt in \cite{Sc1967}, we define the 
\emph{height} of $V$ by
\[
 H(V)=\prod_\pv \norm{\uX}_\pv^{d_\pv/d},
\]
where the product runs over all places $\pv$ of $K$.  This is independent of the choice 
of the basis, and so is the product
\[
 D^*_\uu(V) := H(V)\dist(\uu,V)^{d_\pw/d} 
   = \Big(\frac{\norm{\uu\wedge\uX}_\pw}{\norm{\uu}_\pw}\Big)^{d_\pw/d}
      \prod_{\pv\neq\pw} \norm{\uX}_\pv^{d_\pv/d},
\]
which, in contrast to $\dist(\uu,V)$, does not require division by $\norm{X}_\pw$.   

\begin{definition}
\label{dr:def:omega}
We denote by $\homega_{m-1}(\uu)$ (resp.~$\omega_{m-1}(\uu)$) the supremum of all
$\omega\ge 0$ such that, for each sufficiently large value of $Q\ge 1$ 
(resp.~ for arbitrarily large values of $Q\ge 1$), there is a subspace $V$ of $K_\pw^n$
of dimension $m$ defined over $K$ with
\[
 H(V)\le Q \et D^*_\uu(V)\le Q^{-\omega}.
\]  
\end{definition}

When $K=\bQ$ and $\bQ_\pw=\bR$, this agrees with \cite[Definition~2.1]{R2016} and
\cite[Definition~3]{BL2010}.  In that case, the formulas for $H(V)$ and $D^*_\uu(V)$ 
simplify to $H(V)=\norm{\uX}$ and $D^*_\uu(V)  = \norm{\uu\wedge\uX}/\norm{\uu}$
with the usual Euclidean norms, by choosing $\uX=\ux_1\wedge\cdots\wedge\ux_m$ 
for a basis $(\ux_1,\dots,\ux_m)$ of the $\bZ$-module $V\cap \bZ^n$.

In \cite[Definition 10.5]{PR2023}, one finds an alternative description of 
$\homega_{m-1}(\uu)$ and $\omega_{m-1}(\uu)$, inline with 
\cite[Section 4, Proposition]{BL2010}.  The proof of \cite[Proposition 10.6]{PR2023} 
shows the equivalence of the two definitions based on an adelic version of
Mahler's theory of compound bodies. 

\begin{definition}
\label{dr:def:spectrum}
The \emph{spectrum} of a subsequence $(\theta_1,\dots,\theta_\ell)$ of
$(\homega_0,\dots,\homega_{n-2}, \omega_0, \dots, \omega_{n-2})$ relative 
to the pair $K\subset K_\pw$ is the subset of\/ $[0,\infty]^\ell$ made of the points
\[ 
  (\theta_1(\uu),\dots,\theta_\ell(\uu))
\]
for all $\uu\in K_\pw^n$ with linearly independent coordinates over $K$.
\end{definition}

The next result is a direct consequence of \cite[Proposition 10.6]{PR2023}.  We 
explain this in Section \ref{ssec:pgn:tool}, in stronger form, using tools from
parametric geometry of numbers.

\begin{proposition}
\label{dr:prop}
The \emph{spectrum} of a subsequence $(\theta_1,\dots,\theta_\ell)$ of
$(\homega_0,\dots,\homega_{n-2}, \omega_0, \dots, \omega_{n-2})$ relative 
to the pair $K\subset K_\pw$ is the same as its spectrum relative to the pair
$\bQ\subset \bR$.
\end{proposition}

%
%

\subsection{Main results}  Let $n\ge 2$ be an integer, let $K$ be a number field, 
and let $\pw$ be a place of $K$.  The main result of \cite{R2016} is the determination 
of the spectrum of $(\omega_0,\dots,\omega_{n-2})$ relative to the standard extension
$\bQ \subset \bR$ (see \cite[Theorems 2.2 and 2.3]{R2016}).   By Proposition~\ref{dr:prop}, 
it has the same spectrum relative to the extension $K\subset K_\pw$.  This yields the
following statement.
 
\begin{theorem}
\label{dr:thm:omega}
The spectrum of $(\omega_0,\dots,\omega_{n-2})$ relative to the pair
$K\subset K_\pw$ consists of all points $(\beta_0,\dots,\beta_{n-2})\in[0,\infty]^{n-1}$
satisfying $\beta_0\ge 1/(n-1)$ and 
\[
 \frac{j\beta_j}{\beta_j+j+1} \le \beta_{j-1} \le \frac{(n-1-j)\beta_j-1}{n-j}
\]
for each index $j$ with $1\le j\le n-2$, with the understanding that $j\beta_j/(\beta_j+j+1)$
is equal to $j$ when $\beta_j=\infty$.
\end{theorem}

Similarly, Marnat and Moshchevitin determined the spectra of the pairs
$(\homega_0,\omega_0)$ and $(\homega_{n-2},\omega_{n-2})$ relative to
the embedding $\bQ\subset \bR$, in \cite[Theorem 1]{MM2020}.  In combination 
with Proposition~\ref{dr:prop}, we deduce the following.

\begin{theorem}
\label{dr:thm:MM}
{\rm (i)} The spectrum of\/ $(\homega_0,\omega_0)$ relative to 
$K\subset K_\pw$ consists of the points $(\alpha,\infty)$ with $\alpha\in[1/(n-1),\, 1]$
and of all $(\alpha,\beta)\in(0,\infty)^{2}$ with $\alpha\le \beta$ and
\begin{equation}
\label{dr:thm:MM:eq1}
 n-1 \ge \frac{1}{\alpha} \ge \sum_{i=0}^{n-2}\Big(\frac{\alpha}{\beta}\Big)^i.
\end{equation}
{\rm (ii)} The spectrum of\/ $(\homega_{n-2},\omega_{n-2})$ relative to 
$K\subset K_\pw$ consists  of the points $(\alpha,\infty)$ with $\alpha\in[n-1,\infty]$
and of all $(\alpha,\beta)\in(0,\infty)^{2}$  with $\alpha\le \beta$ and  
\begin{equation}
\label{dr:thm:MM:eq2}
 n-1 \le \alpha \le \sum_{i=0}^{n-2}\Big(\frac{\beta}{\alpha}\Big)^i.
\end{equation}
\end{theorem}

Note that the necessary condition $\alpha\le \beta$ follows directly from 
\eqref{dr:thm:MM:eq1} or from \eqref{dr:thm:MM:eq2} when $n\ge 3$.  
However, it is needed for $n=2$ because, in that case, \eqref{dr:thm:MM:eq1} 
and \eqref{dr:thm:MM:eq2} reduce to $\alpha=1$, and the 
spectrum of $(\homega_0,\omega_0)$ is $\{1\}\times[1,\infty]$.

We will give an alternative proof of Theorem \ref{dr:thm:MM} based only on parametric geometry 
of numbers following the approach of the first author in \cite[Chapter 2]{Ri2019}.
As Schmidt and Summerer note in \cite[Page 85]{SS2013b} for the case $n=4$, the inequalities 
\eqref{dr:thm:MM:eq1} and \eqref{dr:thm:MM:eq2} can be obtained from each other by replacing
$\alpha$ with $1/\alpha$, $\beta$ with $1/\beta$ and by reversing the inequalities.
They view that as a manifestation of the duality between the two pairs of exponents.
We managed to reflect this duality in our arguments as well.

Working over the extension $\bQ\subset\bR$, the first author exhibited a large 
subset of the spectrum of a general pair in \cite[Theorem 2.1.15]{Ri2019}.  With 
the convention that $\infty^0=1$, it can be described as follows.

\begin{proposition}
\label{dr:prop:Ri}
Let $m$ be an integer with $1\le m\le n-1$.  The spectrum $\cS_{m,n}$
of the pair $(\homega_{m-1},\omega_{m-1})$ relative to $K\subset K_\pw$
contains the points $(\alpha,\infty)$ with $\alpha\in[m/(n-m),\infty^{m-1}]$
and the points $(\alpha,\beta)\in (0,\infty)^2$ with $\alpha\le \beta$,
\[
 \frac{m}{n-m}\le \alpha 
 \et
 \alpha\sum_{i=0}^{n-m-1}\Big(\frac{\alpha}{\beta}\Big)^i 
              \le \sum_{i=0}^{m-1}\Big(\frac{\beta}{\alpha}\Big)^i.
\]
Equivalently, it contains the pairs $(\alpha,\beta)\in [0,\infty]^2$ which, for some 
$g\in[1,\infty]$, satisfy 
\[
 \alpha = \Big(\sum_{i=0}^{m-1}g^i\Big)/\Big(\sum_{i=0}^{n-m-1}1/g^i\Big) 
 \et 
 \beta\ge g\alpha.
\]
\end{proposition}

The equivalence between the two formulations is a particular case of the following
observation.  If $\varphi\colon(0,\infty)\to(0,\infty)$ is a continuous, increasing
and unbounded map, then a point $(\alpha,\beta)\in(0,\infty)^2$ satisfies
$\varphi(1)\le \alpha\le \varphi(\beta/\alpha)$ if and only there exists $g\in[1,\infty)$
such that $\alpha=\varphi(g)$ and $\beta\ge g\alpha$.

For $m=1$ and for $m=n-1$, Theorem \ref{dr:thm:MM} shows that $\cS_{m,n}$ 
contains no more points than those given by Proposition \ref{dr:prop:Ri}.  
Based on this, the first author conjectured that this is 
true for all $m$ with $1\le m\le n-1$, in \cite[Conjecture 2.1.4]{Ri2019}.  He also
proved this in dimension $n=4$ for the remaining integer $m=2$, yielding
the following description of $\cS_{2,4}$.

\begin{theorem}
\label{dr:thm:Ri}
We have \ 
$\disp
 \cS_{2,4} = \big\{ (\alpha,\beta)\in [0,\infty]^2 \,;\, 1\le \alpha^2\le \beta \big\}$\,.
\end{theorem} 

Unfortunately, this conjecture already breaks in dimension $n=5$ for $\cS_{3,5}$.  
As we will see, this particular spectrum is complicated.  Our most satisfactory result 
about it is the following.

\begin{theorem}
\label{dr:thm:S35}
The points $(\alpha,\beta)$ of\/ $\cS_{3,5}$ with $\alpha\ge 5$ 
are the pairs $(\alpha,\infty)$ with $\alpha\in[5,\infty]$ and the pairs $(\alpha,\beta)
\in [5,\infty)^2$ with
\[
 5 \le \alpha \le g^2 +1 \quad\text{where}\quad g=\beta/\alpha.
\]
Equivalently, they are the pairs $(\alpha,\beta)\in [0,\infty]^2$ which, for some 
$g\in[2,\infty]$, satisfy 
\begin{equation*}
\label{dr:thm:S35:eq}
 \alpha = g^2+1
 \et 
 \beta\ge g\alpha.
\end{equation*}
\end{theorem} 

Thus the boundary of $\cS_{3,5}$ contains the arc of curve $g \mapsto (g^2+1)(1,g)$ 
with $g\ge 2$.  This contradicts \cite[Conjecture~2.1.4]{Ri2019} since this curve lies 
outside of the subset of $\cS_{3,5}\cap\bR^2$ given by Proposition~\ref{dr:prop:Ri}, 
namely the set of points $(\alpha,\beta)\in (0,\infty)^2$ with
\begin{equation}
\label{dr:eq:conj}
 \frac{3}{2} \le \alpha \le \frac{g^2+g+1}{1+1/g} = g^2+1-\frac{1}{g+1} 
 \quad \text{where} \quad  g=\beta/\alpha.
\end{equation}  
In section \ref{sec:more}, we will also show that the boundary of $\cS_{3,5}$ contains the arc of curve
\begin{equation}
\label{dr:eq:arc}
 g\mapsto ((3/2)g^2-g+1)(1,g) \quad \text{with $1.798\le g\le 1.839$,}
\end{equation} 
which also lies outside of the set described by \eqref{dr:eq:conj}.  
In general, we have the following result.

\begin{theorem}
\label{dr:thm:rectangles}
Let $m$ be an integer with $1\le m\le n-1$.  The spectrum
$S_{n-m,n}$ is a closed subset of $[(n-m)/m,\infty]^2$ with the property that
\[
 [(n-m)/m,\alpha] \times [\beta,\infty] 
   \subseteq S_{n-m,n}
\]
for each $(\alpha,\beta)\in S_{n-m,n}$.   
\end{theorem}

Thus, for such $m$, there exist a closed subinterval $I$ of $[(n-m)/m,\infty]$
and a non-decreasing left-continuous function 
$F\colon I \to[(n-m)/m,\infty]$ such that $S_{n-m,n}$ consists of the pairs 
$(\alpha,\beta)$ with $\alpha\in I$ and $F(\alpha)\le \beta\le\infty$.
If $m\ge 2$, we have $I=[(n-m)/m,\infty]$ by Proposition~\ref{dr:prop:Ri}
while, if $m=1$, Theorem~\ref{dr:thm:MM} (i) shows that $I=[1/(n-1),1]$.  
It would be interesting to know if $F$ is continuous.

We prove Proposition~\ref{dr:prop:Ri} and Theorem~\ref{dr:thm:rectangles}
in section~\ref{sec:cons}, Theorem~\ref{dr:thm:MM} in section~\ref{sec:MM},
Theorem~\ref{dr:thm:Ri} in section~\ref{sec:Ri}, and Theorem~\ref{dr:thm:S35}
in section~\ref{sec:dim5}.  Section~\ref{sec:more} provides complementary
results about $\cS_{3,5}$, and we conclude in section~\ref{sec:fin} with 
two final remarks.  In the next sections \ref{sec:pgn}, \ref{sec:part} and
\ref{sec:type}, we recall the necessary results of parametric geometry of
numbers, and develop tools for the present study.

%
%

\section{Link with parametric geometry of numbers}
\label{sec:pgn}
Parametric geometry of numbers was introduced by Schmidt and Summerer
in \cite{SS2009, SS2013a} to study rational approximation to points in $\bR^n$, 
then completed by Roy in \cite{R2015}, and extended 
to deal with approximation over a number field $K$ by Poëls and Roy in \cite{PR2023}.
To a non-zero point $\uu\in K_\pw^n$, for a place $\pw$ of $K$, the latter 
research attaches a family of convex bodies $\cC_\uu(q)$ in $\bA_K^n$ with 
parameter $q\in[0,\infty)$, where $\bA_K$ stands for the ring of adeles of $K$
(see \cite[Section 7]{PR2023}).   The logarithms of the $n$ successive minima of 
$\cC_\uu(q)$ with respect to $K^n$ yield a function of $q$ with
values in $\bR^n$ which contains much information about the point $\uu$.  In
particular, the exponents $\homega_{m-1}(\uu)$ and $\omega_{m-1}(\uu)$ 
for $m=1,\dots,n-1$ can be computed, in a simple way, in terms of the behaviour 
at infinity of the said function or of any approximation $\uL\colon[0,\infty)\to\bR^n$ 
of it with bounded difference.  The main result \cite[Theorem A]{PR2023} generalizing
that of \cite{R2015} is that, modulo the additive group of bounded functions, the set 
of those maps $\uL$ is represented by a set of functions $\uP\colon[0,\infty)\to\bR^n$ 
called $n$-systems, characterized simply by combinatorial properties.   This 
translates the problem of computing the spectrum of exponents like those 
considered here, in terms of $n$-systems. 

%
%

\subsection{Main definitions}
\label{ssec:pgn:def}
There are many ways of defining an $n$-system.  Below, we recall that from  
\cite[Section 2.5]{R2015}, where it is called an $(n,0)$-system, as it is most 
convenient for our purpose.   In order to unify the analysis of the spectra of 
$(\homega_{m-1},\omega_{m-1})$ for $m=1$ and $m=n-1$, 
we begin by introducing a larger class of functions. 

\begin{definition}
 \label{pgn:def:signfree}
Let $I$ be a closed interval of $\bR$ with non-empty interior.  A \emph{sign-free-$n$-system} 
on $I$ is a map $\uP=(P_1,\dots,P_n)\colon I \to\bR^n$ with the following properties:
\begin{itemize}
 \item[(S1)] $P_1,\dots,P_n$ are continuous and piecewise affine on $I$
  with slopes in $\{0, 1\}$;
 \item[(S2)] $P_1(q)\le\cdots\le P_n(q)$ and $P_1(q)+\cdots+P_n(q)=q$ for
  each $q\in I$;
 \item[(S3)]  if, for some $j\in\{1,\dots,n-1\}$, the sum  $P_1+\cdots+P_j$ changes
  slope from $1$ to $0$ at some interior point $q$ of $I$, then $P_j(q)=P_{j+1}(q)$.
\end{itemize}
We say that it is an \emph{$n$-system} if $P_1(q)\ge 0$ for each $q\in I$.  
\end{definition}

To clarify this notion, we recall that a function $f\colon I\to\bR$, defined on a closed
interval $I$ of $\bR$  with non-empty interior, is continuous and piecewise affine if
its graph in $I\times\bR$ 
is a closed polygonal line.  Then, the \emph{slopes} of $f$ are the slopes of the 
line segments that compose this polygonal line, or equivalently the values of the 
derivative of $f$ on the open subintervals of $I$ over which the restriction of $f$ is
affine.

If a map $\uP=(P_1,\dots,P_n)\colon I \to\bR^n$ satisfies conditions (S1) and
(S2) of Definition \ref{pgn:def:signfree}, then, by (S1), the interval $I$ partitions
into maximal closed subintervals over which each component $P_1,\dots,P_n$ of
$\uP$ is affine of constant slope $0$ or $1$, and, by (S2), these slopes add to $1$.  
Thus, over such an interval, exactly one component has slope $1$ while the others 
have slope $0$.  Consequently, for each $j=1,\dots,n$, the sum
$P_1+\cdots+P_j\colon I \to \bR$ is continuous and piecewise affine with slopes
in $\{0,1\}$.  At an interior point of $I$, such a sum can only change slope from 
$0$ to $1$ or from $1$ to $0$.  Condition (S3) imposes a restriction to $\uP$ when 
the second case applies. 

Note that condition (S2) implies that an $n$-system takes values in $[0,\infty)^n$, so 
its domain is contained in $[0,\infty)$.  Although we will not need that here, it is an 
amusing exercise to show that any sign-free-$n$-system on $[0,\infty)$ can be 
approximated up to bounded difference by an $n$-system on $[0,\infty)$.  However, 
there is more flexibility in the domain of a sign-free-$n$-system, which could even 
be all of $\bR$.  

\begin{definition}
 \label{pgn:def:combined}
The \emph{combined graph} of a sign-free-$n$-system $\uP=(P_1,\dots,P_n)\colon I\to \bR^n$
is the union of the graphs of $P_1,\dots,P_n$ in $I\times\bR$. 
\end{definition}

Although condition (S2) requires that a sign-free-$n$-system $\uP$ takes values in the set 
\[
 \Delta_n=\{(x_1,\dots,x_n)\in\bR^n\,;\, x_1\le \cdots\le x_n\},
\]
this is in general not enough to recover $\uP$ from its combined graph.  We 
need additional restrictions.

\begin{definition}
 \label{pgn:def:switch}
Let $\uP=(P_1,\dots,P_n)$ be a sign-free-$n$-system on some closed interval $I$ 
of $\bR$ with non-empty interior.  A \emph{switch number} of $\uP$ is
a point of $I$ on the boundary of $I$, or an interior point $q$ of $I$ at which at least 
one of the sums $P_1+\cdots+P_j$ with $1\le j<n$ changes slope from $0$ to $1$.
A \emph{switch point} of $\uP$ is the value $(P_1(q),\dots,P_n(q))$ of $\uP$
at a switch number $q$.
Given a number $\delta>0$, we say that $\uP$ is \emph{rigid of mesh $\delta$} if it
admits at least one switch number and if 
$P_1(q),\dots,P_n(q)$ are $n$ distinct non-zero multiples of $\delta$ for each switch
number $q$ of $\uP$.  We say that $\uP$ is \emph{non-degenerate} if
$P_1(q)<\cdots<P_n(q)$ at each interior point $q$ of $I$ where $\uP$ is differentiable.
\end{definition}

The above notion of non-degeneracy for an $n$-system is the same as the one 
introduced by the first author in \cite[Definition 1.2.1 (S4)]{Ri2019}.  However, we
warn the reader that it is slightly weaker than the notion of non-degeneracy from
\cite[Section 2]{R2016} which asks furthermore that $P_1(q)<\cdots<P_n(q)$ 
at each point $q\in I$ on the boundary of $I$.

The next lemma uses the continuous map 
$\Phi_n\colon\bR^n\to\Delta_n$ 
which lists the coordinates of a point in non-decreasing order
(cf.\ \cite[Section~1]{R2015}).

\begin{lemma}
 \label{pgn:lemma:combined_graph}
Let $\uP=(P_1,\dots,P_n)$ be a sign-free-$n$-system on some closed interval $I$ 
of $\bR$ with non-empty interior, and let $J$ be a closed subinterval of $I$ with
non-empty interior and no switch number in its interior.  Then, for each $a\in J$, there 
exists an index $k$ depending on $a$ (but not necessarily unique) such that, 
for each $q\in J$, we have
\begin{equation}
 \label{pgn:lemma:combined_graph:eq}
 \uP(q) = \Phi_n(P_1(a),\dots,\widehat{P_k(a)},\dots,P_n(a),P_k(a)+q-a)
\end{equation}
where the hat on $P_k(a)$ means that this coordinate is omitted from the list. If, for some
$a\in J$, we have $P_1(a)<\cdots<P_n(a)$, then the restriction of\/ $\uP$ to $J$ is 
non-degenerate.
\end{lemma}

\begin{proof}
Geometrically, the first assertion means that the combined graph of $\uP$ over the 
interval $J$ consists of the horizontal line segments $J\times P_i(a)$ for each index $i$
with $1\le i\le n$ and $i\neq k$, together with the line segment of slope $1$ passing 
through the point $(a,P_k(a))$ and projecting down to $J$.  For the proof, we first 
observe that, for a given $a\in J$, the formula \eqref{pgn:lemma:combined_graph:eq}
holds locally in $J\cap[a-\epsilon,a+\epsilon]$ for a sufficiently small $\epsilon>0$,
by choosing $k$ to be the index $s$ for which $P_s$ has right derivative $1$ at $a$, if 
$a$ is not a right end point of $J$, or the index $r$ for which $P_r$ has left derivative 
$1$ at $a$, if $a$ is not a left end point of $J$.  If $a$ is an interior point of $J$, then either 
choice is good because we then have $r\le s$ and $P_r(a)=\dots=P_s(a)$ since 
$a$ is not a switch number of $\uP$ (as illustrated in \cite[Figure 1]{R2016} 
when $r<s$).  The first assertion follows from this because
the functions defined by the right-hand side of \eqref{pgn:lemma:combined_graph:eq}
for distinct values of $a$ and for their associated indices $k$ agree on $J$ as soon as
they agree on a non-empty open subinterval of $J$.  The second assertion of the lemma
in turn follows from formula \eqref{pgn:lemma:combined_graph:eq}.
\end{proof}

\begin{corollary}
 \label{pgn:def:cor}
Let $\uP$ and $I$ be as in the above lemma.  If\/ $\uP$ is rigid of mesh $\delta$ for some
$\delta>0$, then $\uP$ is non-degenerate.  If\/ $\uP$ is non-degenerate,
then it is uniquely determined by its combined graph in $I\times\bR$.
\end{corollary}

\begin{proof}
Suppose first that $\uP$ is rigid of mesh $\delta$ for some $\delta>0$.  Then its 
domain $I$ partitions into maximal closed subintervals whose interior is non-empty
but does not contain any switch number of $\uP$. Since $\uP$ admits at least one
switch number, these intervals are of the form  $(-\infty,a]$, $[a,b]$ or
$[a,\infty)$ where $a$ and $b$ are switch numbers.  In particular, they contain at
least one switch number and, by the lemma, the restriction of $\uP$ to any of these 
subintervals is non-degenerate.  So $\uP$ itself is non-degenerate.
Finally, if $\uP$ is non-degenerate, then it is determined by its combined 
graph since $\uP$ is continuous and the interior points of $I$ where it is not 
differentiable form a discrete subset of $I$.  
\end{proof}

Similarly, one shows that $\uP$ is non-degenerate 
if it admits at least one switch number in the interior of $I$ and if its components 
take $n$ distinct values at each such number.   We conclude 
by recalling two more important classes of $n$-systems.

\begin{definition}
 \label{pgn:def:systems}
We say that an $n$-system $\uP=(P_1,\dots,P_n)$ is \emph{proper} if it is defined on
$[q_0,\infty)$ for some $q_0\ge 0$ and if $P_1(q)$ goes to infinity with $q$.  We say 
that it is \emph{self-similar} if moreover there exists some $\rho>1$ such that 
$\uP(\rho q)=\rho\uP(q)$  for each $q\ge q_0$.  
\end{definition}

%
%

\subsection{A geometric characterization}
\label{ssec:pgn:char}

By Corollary \ref{pgn:def:cor}, a non-degenerate sign-free-$n$-system $\uP$ on an interval
$[a,b]$ is uniquely determined by its combined graph.  Conversely, given a graph in 
$[a,b]\times\bR$, it is easy to check if it is the combined graph of such a map.  It suffices to
verify that $[a,b]$ can be partitioned into subintervals $[q_0,q_1],\dots,[q_{N-1},q_N]$ for 
some integer $N\ge 1$ so that the following properties hold:
\begin{itemize}
\item over each subinterval $[q_{i-1},q_i]$, the graph consists of exactly $n-1$ distinct 
  horizontal line segments and one line segment $\Gamma_i$ of slope $1$, each projecting 
 down to $[q_{i-1},q_i]$;
\item for each $i$ with $1\le i\le N-1$, the right end-points of the line segments above 
 $[q_{i-1},q_i]$ are the same, as a set, as the left end-points of those above $[q_i,q_{i+1}]$,
 and the right end-point of $\Gamma_i$ does not lie strictly below the left end-point 
of $\Gamma_{i+1}$ (but can be the same);
 \item the sum of the ordinates of the points above $a$ is $a$ 
\end{itemize}
(cf.~\cite[Section 2.6]{PR2023} or \cite[Section 2]{RR2020}).  The switch numbers of $\uP$ are
then $q_0=a$, $q_N=b$ and the numbers $q_i$ with $1\le i\le N-1$ for which
the right end-point of $\Gamma_i$ lies strictly above the left end-point of $\Gamma_{i+1}$.

%
%

\subsection{Main tool}
\label{ssec:pgn:tool}
For each proper $n$-system $\uP=(P_1,\dots,P_n)$ and each $m=1,\dots,n-1$, we set
\begin{equation}
\label{pgn:eq:chi}
 \chibot_m(\uP)=\liminf_{q\to\infty}\frac{P_{m+1}(q)+\cdots+P_n(q)}{P_1(q)+\cdots+P_m(q)},
 \quad
 \chitop_m(\uP)=\limsup_{q\to\infty}\frac{P_{m+1}(q)+\cdots+P_n(q)}{P_1(q)+\cdots+P_m(q)}.
\end{equation}
These are elements of $[0,\infty]=[0,\infty)\cup\{\infty\}$, a set that we endow with 
the topology of one point compactification of $[0,\infty)$.  Then the map $\iota\colon
[0,\infty]\to[0,1]$ which sends $\infty$ to $0$ and any $x\in[0,\infty)$ to $1/(x+1)$ is a 
homeomorphism. For each integer 
$m\ge 1$, we also endow $[0,\infty]^m$ with the product topology, which makes it Hausdorff 
and compact.  The next result is the central tool on which this paper relies.  

\begin{proposition}
\label{pgn:prop:allexp}
The set of points
\[
 ( \homega_0(\uu),\dots,\homega_{n-2}(\uu),
     \omega_0(\uu), \dots, \omega_{n-2}(\uu) )
\]
where $\uu\in K_\pw^n$ has $K$-linearly independent coordinates is a closed subset of 
$[0,\infty]^{2n-2}$ which coincides with the set of points
\begin{equation}
\label{pgn:prop:allexp:eq}
 \big(\chibot_{n-1}(\uP), \dots, \chibot_{1}(\uP), \chitop_{n-1}(\uP), \dots, \chitop_1(\uP)\big)
\end{equation}
where $\uP$ is a proper $n$-system or even a proper non-degenerate $n$-system.  
Moreover, the points coming from self-similar $n$-systems are dense in that set.
\end{proposition}

\begin{proof}
By \cite[Proposition~7.1]{R2015} (applied with $\gamma=0$), any proper 
$n$-system $\uP$ differs by a bounded function from some proper rigid $n$-system $\uR$ of 
mesh $1$, an $n$-system which, by Corollary~\ref{pgn:def:cor}, is non-degenerate.
As the point \eqref{pgn:prop:allexp:eq} does not change when we replace $\uP$ by $\uR$, 
we obtain the same set of points by letting $\uP$ run through all proper $n$-systems 
in \eqref{pgn:prop:allexp:eq}, or through the smaller set of proper non-degenerate 
$n$-systems. Thus, it suffices to prove the first part of the proposition for proper $n$-systems.

Definition~1.1 in \cite{R2017} attaches to each linear map 
$T=(T_1,\dots,T_\ell)\colon\bR^n\to\bR^\ell$,
a spectrum of approximation to points in $\bR^n$, denoted $\image^*(\mu_T)$.  
By \cite[Theorem~3.1]{R2017}, it consists of the points 
\[
 \mu_T(\uP)
  =\Big(
    \liminf_{q\to\infty}\frac{T_1(\uP(q))}{q},
    \dots,
    \liminf_{q\to\infty}\frac{T_\ell(\uP(q))}{q}
    \Big)
\]
where $\uP$ is a proper $n$-system.  By \cite[Corollary 8.5]{R2017}, this set 
$\image^*(\mu_T)$ is a compact subset 
of $\bR^\ell$, and by \cite[Corollary 7.2]{R2017}, the points $\mu_T(\uP)$ where $\uP$ is a 
self-similar $n$-system are dense in it.  We take $T=(-T_1,\dots,-T_{n-1},T_1,\dots,T_{n-1})$
where, for $m=1,\dots,n-1$, the linear form $T_m$ maps each point of $\bR^n$ to the sum of its 
first $n-m$ coordinates.  We also denote by $\sigma$ the linear automorphism of $\bR^{2(n-1)}$ 
which sends $(\ux,\uy)$ to $(-\ux,\uy)$ for any $\ux,\uy\in\bR^{n-1}$.  Then, the set 
$\cS=\sigma(\image^*(\mu_T))$ consists of the points
\[
 \sigma(\mu_T(\uP))
  =\big(
    \psitop_{n-1}(\uP), \dots, \psitop_1(\uP), \psibot_{n-1}(\uP),\dots, \psibot_1(\uP)
    \big)
\]
where $\uP=(P_1,\dots,P_n)$ is a proper $n$-system and 
\[
 \psibot_m(\uP)=\liminf_{q\to\infty}\frac{P_1(q)+\cdots+P_m(q)}{q},
 \quad
 \psitop_m(\uP)=\limsup_{q\to\infty}\frac{P_1(q)+\cdots+P_m(q)}{q},
\]
for $m=1,\dots,n-1$.  Since the coordinates of $\uP(q)$ sum to $q$, this set $\cS$ 
is a compact subset of $[0,1]^{2(n-1)}$.  By \cite[Proposition 10.6]{PR2023}, $\cS$ is also 
the set of points 
\[
 \big( \iota(\homega_0(\uu)),\dots,\iota(\homega_{n-2}(\uu)),
         \iota(\omega_0(\uu)), \dots, \iota(\omega_{n-2}(\uu)) \big)
\]
where $\iota\colon[0,\infty]\to[0,1]$ is the homeomorphism defined just before the proposition
and where $\uu$ runs through the set of points of $K_\pw^n$ with linearly independent 
coordinates over $K$.  The conclusion follows since 
\[
 \psibot_m(\uP)=\iota(\chitop_m(\uP))
 \et
 \psitop_m(\uP)=\iota(\chibot_m(\uP))
\]
for any $m=1,\dots,n-1$ and any proper $n$-system $\uP$. 
\end{proof}

In particular, the above result shows that the \emph{spectrum} of 
$(\homega_0,\dots,\homega_{n-2}, \omega_0, \dots, \omega_{n-2})$ relative 
to the pair $K\subset K_\pw$ is independent of the choice of $K$ and $\pw$.
This in turn implies Proposition \ref{dr:prop}.   Another direct consequence is
the following statement which is crucial for the present study.

\begin{corollary}
\label{pgn:tool:cor}
Let $m$ be an integer with $1\le m\le n-1$.  The spectrum $\cS_{n-m,n}$ of\/ 
$( \homega_{n-m-1}, \omega_{n-m-1} )$ relative to the pair $K\subset K_\pw$
is the closed subset of $[0,\infty]^2$ made of the points
\[
 \big(\chibot_{m}(\uP), \chitop_{m}(\uP)\big)
\]
where $\uP$ is a proper non-degenerate $n$-system.  The points coming from
self-similar $n$-systems are dense in that set.
\end{corollary}

%
%

\subsection{Duality}
\label{ssec:pgn:duality}
The above corollary describes, in terms of $n$-systems, the spectrum of a pair 
$( \homega_{n-m-1}, \omega_{n-m-1} )$ with $1\le m\le n-1$ .  
We will show below that the spectrum of the dual pair 
$( \homega_{m-1}, \omega_{m-1} )$ can be computed using essentially 
the same formulas in terms of another class of sign-free-$n$-systems which 
we call backwards $n$-systems, a convenient alternative to the dual 
$n$-systems introduced in \cite[Section 2.6]{PR2023}.    In section \ref{sec:fin}, 
we illustrate, in the classical setting where $K=\bQ$ and $K_\pw=\bR$, 
how this is naturally connected with Mahler's duality for convex bodies.  
We start with the following observation.

\begin{proposition}
\label{pgn:duality:prop1}
Let $I$ be a closed subinterval of\/ $\bR$ with non-empty interior and let
$\uP=(P_1,\dots,P_n)\colon I \to \bR^n$ be a sign-free-$n$-system on $I$.
Then $-I=\{-q\,;\, q\in I\}$ is a closed subinterval of $\bR$ with 
non-empty interior and the map $\uP^\vee=(P_1^\vee,\dots,P_n^\vee)
\colon -I \to \bR^n$ given by
\[
 \uP^\vee(q)=(-P_n(-q),\dots,-P_1(-q))
\]
for each $q\in-I$ is a sign-free-$n$-system on $-I$.
\end{proposition}

\begin{proof}
For each $j=1,\dots,n$, the $j$-th component $P_j^\vee$ of $\uP^\vee$ 
is given by 
\[
 P_j^\vee(q)=-P_{n-j+1}(-q)
\] 
for each $q\in -I$.  It is
continuous and, for each closed subinterval $J$ of $I$ of positive length on
which $P_{n-j+1}$ is affine, the restriction of $P_j^\vee$ to $-J$ is affine 
of the same slope.  Thus,  $P_j^\vee$ is piecewise affine with slopes 
in $\{0,1\}$.  This shows that $\uP^\vee$ meets condition (S1) in 
Definition~\ref{pgn:def:signfree}.  Condition (S2) is immediate since, for
each $q\in-I$, we have $-q\in I$, thus
\[
 -P_n(-q)\le \dots\le -P_1(-q)
 \et
 -P_n(-q)-\dots-P_1(-q)=-(-q)=q.
\]
To check condition (S3), suppose that, 
for some $j\in\{1,\dots,n-1\}$, the sum  $P_1^\vee+\cdots+P_j^\vee$ 
changes slope from $1$ to $0$ at an interior point $q$ of $-I$.   Then
there exists $\epsilon>0$ such that  $P_1^\vee+\cdots+P_j^\vee$ is
affine of slope $1$ on $[q-\epsilon,q]$ and constant on $[q,q+\epsilon]$.
Since 
\[
 P_1^\vee(t)+\cdots+P_j^\vee(t)=t+P_1(-t)+\cdots+P_{n-j}(-t)
\]
for each $t\in -I$, it follows that  $P_1+\cdots+P_{n-j}$ is constant on
$[-q,-q+\epsilon]$ while it has slope $1$ on $[-q-\epsilon,-q]$. Thus, the last 
sum changes slope from $1$ to $0$ at $-q$, and so $P_{n-j}(-q)=P_{n-j+1}(-q)$,
which translates into $P_j^\vee(q)=P_{j+1}^\vee(q)$.
\end{proof}

\begin{definition}
\label{pgn:duality:def1}
With the notation of Proposition \ref{pgn:duality:prop1}, we call 
$\uP^\vee\colon -I \to \bR^n$ the \emph{opposite} of $\uP\colon I \to \bR^n$.
We say that $\uP$ is a \emph{backwards $n$-system} (resp.~\emph{a proper 
backwards $n$-system}) if $\uP^\vee$ is an $n$-system (resp.~a proper $n$-system).
\end{definition}

Clearly, we have $(\uP^\vee)^\vee=\uP$ for each sign-free-$n$-system
$\uP$.  Backwards $n$-systems have their domain contained in $(-\infty,0]$, and 
they admit the following alternative characterization.

\begin{lemma}
\label{pgn:duality:lemma1}
A backwards $n$-system is a sign-free-$n$-system $\uP=(P_1,\dots,P_n)
\colon I \to\bR^n$ such that $P_n(q)\le 0$ for each $q\in I$.  It is 
proper if $I=(-\infty,q_0]$ for some $q_0\le 0$ and if $P_n(q)$ goes to
$-\infty$ with $q$.
\end{lemma}

The invariants \eqref{pgn:eq:chi} attached to proper $n$-systems have natural analogs 
for backwards $n$-system, obtained by letting $q$ tend to $-\infty$ instead of $+\infty$.  Thus, 
for each proper backwards $n$-system $\uP=(P_1,\dots,P_n)$ and each $m=1,\dots,n-1$, 
we set
\begin{equation}
\label{pgn:eq:chi:-infty}
 \chibot_m(\uP)=\liminf_{q\to-\infty}\frac{P_{m+1}(q)+\cdots+P_n(q)}{P_1(q)+\cdots+P_m(q)},
 \quad
 \chitop_m(\uP)=\limsup_{q\to-\infty}\frac{P_{m+1}(q)+\cdots+P_n(q)}{P_1(q)+\cdots+P_m(q)}.
\end{equation}
We view the next statement as an expression of duality.

\begin{proposition}
\label{pgn:duality:prop2}
Let $\uP\colon [q_0,\infty) \to \bR^n$ be a proper $n$-system.
For each $m=1,\dots,n-1$, we have
\[
 \chibot_{n-m}(\uP) = \chitop_{m}(\uP^\vee)^{-1}
 \et
 \chitop_{n-m}(\uP) = \chibot_{m}(\uP^\vee)^{-1}.
\]
\end{proposition}

\begin{proof}
Writing $\uP=(P_1,\dots,P_n)$ and $\uP^\vee=(P_1^\vee,\dots,P_n^\vee)$, this follows 
from the identity
\[
 \frac{P_{n-m+1}(q)+\cdots+P_n(q)}{P_1(q)+\cdots+P_{n-m}(q)}
 = \Big( \frac{P^\vee_{m+1}(-q)+\cdots+P^\vee_n(-q)}%
                   {P_1^\vee(-q)+\cdots+P_{m}^\vee(-q)} \Big)^{-1}
\]
valid for each $q\ge q_0$ large enough so that $0<P_1(q)\le \cdots\le P_n(q)$,
a condition which is equivalent to $P_1^\vee(-q)\le \cdots\le P_n^\vee(-q)<0$.
\end{proof}

Combining this proposition with Corollary \ref{pgn:tool:cor}, we obtain the following.

\begin{corollary}
\label{pgn:duality:cor}
Let $m$ be an integer with $1\le m\le n-1$.  The spectrum $\cS_{m,n}$ of\/ 
$( \homega_{m-1}, \omega_{m-1} )$ relative to the pair $K\subset K_\pw$
is the closed subset of $[0,\infty]^2$ made of the points
\[
 \big(\chitop_{m}(\uP)^{-1}, \chibot_{m}(\uP)^{-1}\big)
\]
where $\uP$ is a proper non-degenerate backwards $n$-system.  
\end{corollary}

%
%

\section{A partition} 
\label{sec:part}
Let $m$ and $n$ be integers with $1\le m\le n-1$. 
In this section, we show how to compute the quantities $\chibot_{m}(\uP)$ 
and $\chitop_{m}(\uP)$, defined in \eqref{pgn:eq:chi} and in \eqref{pgn:eq:chi:-infty},
in terms of a partition of the domain of $\uP$.   We begin with the following observation.

\begin{lemma}
\label{part:lemma:Am}
Let $\uP=(P_1,\dots,P_n)\colon I\to\bR^n$ be a non-degenerate 
sign-free-$n$-system on a closed subinterval $I$ of $\bR$ with non-empty 
interior.  Then
\begin{equation}
 \label{part:lemma:Am:eq}
 \cA_m(\uP)=\{a\in I\,;\, P_m(a)=P_{m+1}(a)\}
\end{equation}
is a closed discrete subset of $I$ (i.e.\ has finite intersection with any bounded 
subinterval of $I$). An interior point $a$ of $I$ belongs to 
$\cA_m(\uP)$ if and only if the sum $S_m^-=P_1+\cdots+P_m$ changes slope 
from $1$ to $0$ around $a$.
\end{lemma} 

\begin{proof}
It suffices to prove the second assertion as it implies the first.  So, let $a$
be an interior point of $I$.  If $a\in \cA_m(\uP)$, then, as $\uP$ is non-degenerate,
it is not differentiable at $a$.  More precisely, around the point $a$, its component 
$P_m$ changes slope from $1$ to $0$, $P_{m+1}$ changes slope from $0$ to $1$,
and all the other components are constant, thus $S_m^-$ changes slope 
from $1$ to $0$ around $a$.  Conversely, by definition of a sign-free-$n$-system, we 
have $a\in\cA_m(\uP)$ if $S_m^-$ changes slope from $1$ to $0$ around $a$.
\end{proof}

Thus, the set $\cA_m(\uP)$ provides a partition of the domain $I$ of $\uP$. 

\begin{definition}
\label{part:def}
For $\uP$ as in Lemma \ref{part:lemma:Am}, we say that the set $\cA_m(\uP)$ 
given by \eqref{part:lemma:Am:eq} is the set of \emph{$m$-division numbers} of 
$\uP$.  We denote by $\cI_m(\uP)$ the set of finite intervals $[a,b]$ with end 
points $a\le b$ in $\cA_m(\uP)$.   We say that an interval in $\cI_m(\uP)$ is \emph{simple} 
if its interior is non-empty and does not contain an element of $\cA_m(\uP)$. 
\end{definition}

Simple intervals are important because of the following property.

\begin{lemma}
\label{part:lemma:simple}
Let $\uP=(P_1,\dots,P_n)$ be a non-degenerate sign-free-$n$-system. 
Suppose that $[a,b]$ is a simple interval in $\cI_m(\uP)$.  Then there exists
$t\in(a,b)$ such that $P_1+\dots+P_m$ is constant on $[a,t]$, and affine of 
slope $1$ on $[t,b]$.  The number $t$ is given by
\begin{equation}
 \label{type:lemma:eq}
 t = P_1(a)+\cdots+P_m(a)+P_{m+1}(b)+\cdots+P_n(b).
\end{equation}
\end{lemma}

\begin{proof}
Since $[a,b]$ is simple, we have $P_m(q)<P_{m+1}(q)$ for each $q\in (a,b)$.  Thus,
the sum $S_m^-=P_1+\cdots+P_m$ cannot change slope from $1$ to $0$ anywhere on 
$(a,b)$.  Moreover, as $\uP$ is non-degenerate with $P_m(a)=P_{m+1}(a)$,
its component $P_{m+1}$ has slope $1$ to the right of $a$.  Similarly, $P_m$ has 
slope $1$ to the left of $b$ because $P_m(b)=P_{m+1}(b)$.  Thus, $S_m^-$ has 
slope $0$ to the right of $a$ and slope $1$ to the left of $b$.  We conclude that there 
exists $t\in (a,b)$ such that $S_m^-$ is constant on $[a,t]$ and affine of slope $1$ 
on $[t,b]$.  Then $P_1,\dots,P_m$ are constant on $[a,t]$, while $P_{m+1},\dots,P_n$ 
are constant on $[t,b]$, thus
\[
 \uP(t)=( P_1(a),\dots,P_m(a),P_{m+1}(b),\dots,P_n(b)),
\]
and so $t$ is given by \eqref{type:lemma:eq}.
\end{proof}

The next result shows that, when $\uP$ is a proper non-degenerate $n$-system, 
its values on $\cA_m(\uP)$ suffice to compute $\chibot_m(\uP)$ and $\chitop_m(\uP)$.

\begin{proposition}
\label{part:prop:n-sys}
Let $\uP=(P_1,\dots,P_n) \colon [q_0,\infty) \to \bR^n$ be a proper 
non-degenerate $n$-system.  Then
\[
 \cA_m(\uP)=\{a\ge q_0\,;\, P_m(a)=P_{m+1}(a)\}
\]
is a closed discrete unbounded subset of $[q_0,\infty)$.  Let $a_1<a_2<a_3<\cdots$ 
be the list of its elements in increasing order.   Then, for each $i\ge 1$, there 
exists $t_i\in(a_i,a_{i+1})$ such that $P_1+\cdots+P_m$ is constant on
$[a_i,t_i]$ and affine of slope $1$ on $[t_i,a_{i+1}]$.  Moreover, we have 
\begin{equation}
 \label{part:prop:n-sys:eq}
 \frac{n-m}{m}\,\le\, \chibot_m(\uP)=\liminf_{i\to\infty}\frac{A_i^+}{A_i^-}
 \,\le\,
 \chitop_m(\uP)=\limsup_{i\to\infty}\frac{A_{i+1}^+}{A_i^-}
\end{equation}
where $A_i^-=P_1(a_i)+\cdots+P_m(a_i)$ and $A_i^+=P_{m+1}(a_i)+\cdots+P_n(a_i)$ 
for each $i\ge 1$.
\end{proposition} 

\begin{proof}
The functions 
\begin{equation}
 \label{part:prop:n-sys:defS}
 S_m^-=P_1+\cdots+P_m \et S_m^+=P_{m+1}+\cdots+P_n
\end{equation}
are continuous and piecewise affine on $[q_0,\infty)$, with slopes 
in $\{0,1\}$, and satisfy
\begin{equation}
 \label{part:prop:n-sys:eq2}
 S_m^-(q)+S_m^+(q)=q
 \et
 P_1(q)\le \frac{1}{m}S_m^-(q)\le P_m(q)\le \frac{1}{n-m}S_m^+(q)\le P_n(q)
\end{equation}
for each $q\ge q_0$.  Since $\uP$ is proper, its component $P_1$ is unbounded 
and so both $S_m^-(q)$ and $S_m^+(q)$ tend to infinity with $q$. 

By Lemma~\ref{part:lemma:Am},  $\cA_m(\uP)$ is a closed discrete 
subset of $[q_0,\infty)$.  It is unbounded because otherwise, by the same
lemma, there would exist $c\ge q_0$ such that $S_m^-$ is affine of slope 
$0$ or $1$ on $[c,\infty)$.  Slope $0$ is excluded because
$S_m^-$ is unbounded.  Slope $1$ is also impossible because, by the first part
of \eqref{part:prop:n-sys:eq2}, that would imply that $S_m^+$ is constant 
on $[c,\infty)$.  Thus, we can list the elements of $\cA_m(\uP)$ as stated in the 
proposition.

By construction, each $[a_i,a_{i+1}]$ is a simple interval in $\cI_m(\uP)$.  So, 
by Lemma~\ref{part:lemma:simple}, there exists $t_i\in(a_i,a_{i+1})$ such 
that $S_m^-$ is constant on $[a_i,t_i]$ and affine of slope $1$ on $[t_i,a_{i+1}]$.  
If $i$ is large enough, $S_m^-$ is positive thus nowhere zero on $[a_i,a_{i+1}]$,
and so the ratio
\begin{equation}
 \label{part:prop:n-sys:defchi}
 \chi_m(q)=\frac{S_m^+(q)}{S_m^-(q)}
\end{equation}
is a continuous function of $q$ on that interval.  As $S_m^+$ is affine of slope $1$ 
on $[a_i,t_i]$, and constant on $[t_i,a_{i+1}]$, the function $\chi_m$ is increasing 
on $[a_i,t_i]$ and decreasing on $[t_i,a_{i+1}]$.  So, we find
\begin{align}
 &\max\{\chi_m(q)\,;\,a_i\le q\le a_{i+1}\} 
   = \frac{S_m^+(t_i)}{S_m^-(t_i)}  
  = \frac{S_m^+(a_{i+1})}{S_m^-(a_i)} 
  = \frac{A_{i+1}^+}{A_i^-},
 \label{part:prop:n-sys:max}
 \\
 &\min\{\chi_m(q)\,;\,a_i\le q\le a_{i+1}\} 
   = \min\left\{\frac{S_m^+(a_i)}{S_m^-(a_i)},\, 
                      \frac{S_m^+(a_{i+1})}{S_m^-(a_{i+1})}\right\}  
  = \min\left\{\frac{A_i^+}{A_i^-},\, 
                      \frac{A_{i+1}^+}{A_{i+1}^-}\right\}. 
 \label{part:prop:n-sys:min} 
\end{align}
This yields \eqref{part:prop:n-sys:eq} upon noting that $\chi_m$ is bounded below
by $(n-m)/m$ by the second part of \eqref{part:prop:n-sys:eq2}.
\end{proof}

We conclude with the following dual statement.

\begin{proposition}
\label{part:prop:bn-sys}
Let $\uP=(P_1,\dots,P_n) \colon (-\infty,q_0] \to \bR^n$ be a proper 
non-degenerate backwards $n$-system.  Then
\[
 \cA_m(\uP)=\{a\le q_0\,;\, P_m(a)=P_{m+1}(a)\}
\]
is a closed discrete unbounded subset of $(-\infty,q_0]$.  Let $a_{-1}>a_{-2}>a_{-3}>\cdots$ 
be the list of its elements in decreasing order.   Then, we have 
\begin{equation}
 \label{part:prop:bn-sys:eq}
 \chibot_m(\uP)=\liminf_{i\to-\infty}\frac{A_{i+1}^+}{A_i^-}
 \,\le\,
 \chitop_m(\uP)=\limsup_{i\to-\infty}\frac{A_i^+}{A_i^-}
 \,\le\,
 \frac{n-m}{m} 
\end{equation}
where $A_i^-=P_1(a_i)+\cdots+P_m(a_i)$ and $A_i^+=P_{m+1}(a_i)+\cdots+P_n(a_i)$ 
for each $i\ge 1$.
\end{proposition} 

\begin{proof}
We proceed as for the previous proposition.  We first note that the functions $S_m^-$ and 
$S_m^+$ given by \eqref{part:prop:n-sys:defS} for the present choice of $\uP$ are 
continuous and piecewise affine on $(-\infty,q_0]$, with slopes in $\{0,1\}$, and that 
they satisfy \eqref{part:prop:n-sys:eq2} for each $q\le q_0$.   Since $\uP$ is proper, $P_n(q)$ 
tends to $-\infty$ with $q$, and so do $S_m^-(q)$ and $S_m^+(q)$. 

Using Lemma~\ref{part:lemma:Am} in the same way, we find that $\cA_m(\uP)$ is a closed 
discrete unbounded subset of $(-\infty,q_0]$, and so its elements can be listed as stated. 
We also deduce from the lemma that, for each $i\le -2$, there exists
$t_i\in(a_i,a_{i+1})$ such that $S_m^-$ is constant on $[a_i,t_i]$ and affine of
slope $1$ on $[t_i,a_{i+1}]$.

If $i$ is small enough, both $S_m^-$ and $S_m^+$ are negative on $[a_i,a_{i+1}]$ ,
so the ratio $\chi_m$ defined by \eqref{part:prop:n-sys:defchi} is continuous 
on $[a_i,a_{i+1}]$ and positive valued.  As $S_m^+$ is affine of slope $1$ 
on $[a_i,t_i]$, and constant on $[t_i,a_{i+1}]$, the function $\chi_m$ is decreasing 
on $[a_i,t_i]$ and increasing on $[t_i,a_{i+1}]$.  So, \eqref{part:prop:n-sys:max}
and \eqref{part:prop:n-sys:min} need to be modified by replacing everywhere $\min$ 
by $\max$, and $\max$ by $\min$.  
This yields \eqref{part:prop:bn-sys:eq} since $\chi_m$ is bounded above
by $(n-m)/m$ by the second part of \eqref{part:prop:n-sys:eq2}.
\end{proof}

%
%

\section{The type of an interval}
\label{sec:type}
Let $m,n\in \bZ$ with $1\le m\le n-1$. 
The following special case of \cite[Theorem 1.3.4]{Ri2019}
compares the values of a sign-free-$n$-system at 
any two $m$-division numbers.   We will show below that it is best possible.

\begin{proposition}
\label{type:prop:kl}
Let $\uP$ be as in Lemma \ref{part:lemma:Am}.  Suppose that $a<b$ are  
elements of $\cA_m(\uP)$.  Then there are integers $k$ and $\ell$ with
$1\le k\le m< \ell\le n$ such that, for each $j=1,\dots,n$, 
\[
 P_j(a) < P_j(b)  \quad \Longleftrightarrow \quad k\le j\le \ell.
\]
Moreover, we have $P_j(a)\le P_{j-1}(b)$ for each $j=k+1,\dots,\ell$.
\end{proposition}

\begin{proof}
Since $\uP$ is non-degenerate, $P_{m+1}$ has slope $1$
to the right of $a$ while $P_m$ has slope $1$ to the left of $b$.  In particular, $P_m$
and $P_{m+1}$ are not constant on $[a,b]$ and so we have 
$P_m(a)=P_{m+1}(a)<P_m(b)=P_{m+1}(b)$.

Let $k$ be the smallest index $j$ for which $P_j(a) < P_j(b)$, and let $\ell$ be the
largest such index.  By the above, we have $1\le k\le m$ and $m+1\le \ell\le n$.  Since 
$P_{j-1}(b)<P_j(b)$ for each value of $j$ except $m+1$, it remains only to show 
that $P_j(a)\le P_{j-1}(b)$ for each $j$ with $k<j\le \ell$ and $j\neq m+1$.
  
Suppose that $j$ is an integer with $m+1< j\le \ell$.  Since 
$P_\ell(a)<P_\ell(b)$, there is a point $c$ in $(a,b)$ around which $P_\ell$ is affine 
of slope $1$.  Since $P_{m+1}$ has slope $1$ to the right of $a$, it follows
that the sum $P_1+\cdots+P_{j-1}$ has slope $1$ to the right of $a$ but
slope $0$ around $c$.  Thus, it changes slope from $1$ to $0$ at some point $q$ in
$(a,c)$.  At that point, we have $P_{j-1}(q)=P_j(q)$ and so
\begin{equation} 
\label{lemma:type:eq}
  P_j(a) \le P_j(q)=P_{j-1}(q) \le P_{j-1}(b).
\end{equation}

Similarly, suppose that $j$ is an integer with $k< j\le m$.  Since 
$P_k(a)<P_k(b)$, there is a point $d$ in $(a,b)$ around which $P_k$ is affine
of slope $1$.  Since $P_m$ has slope $1$ to the left of $b$, the sum 
$P_1+\cdots+P_{j-1}$ is affine of slope $1$ around $d$ and has 
slope $0$ to the left of $b$. Thus, it changes slope from $1$ to $0$ at some point $q$ in
$(d,b)$.  At that point, we have $P_{j-1}(q)=P_j(q)$ and so
\eqref{lemma:type:eq} holds again.
\end{proof}

\begin{definition}
\label{part:def2}
Let $\uP$ be as in Lemma \ref{part:lemma:Am}, and let $[a,b]$ be an interval in $\cI_m(\uP)$,
namely an interval $[a,b]$ with $a,b\in \cA_m(\uP)$ and $a\le b$.  We define the \emph{type}
of $[a,b]$ to be the set of integers $j\in\{1,\dots,n\}$ for which $P_j(a)<P_j(b)$.   
\end{definition}

By Proposition \ref{type:prop:kl}, the type of an interval $[a,b]$ in $\cI_m(\uP)$
with $a<b$ is a sequence of consecutive integers $\{k,\dots,\ell\}$ where $k$ and $\ell$
satisfy $1\le k\le m<\ell\le n$.  For an interval $[a,a]$ reduced to a single point $a$,
it is the empty set.  

Proposition \ref{type:prop:kl} admits the following converse (cf.~\cite[Theorem 1.3.8]{Ri2019}).   

\begin{proposition}
\label{type:prop:converse}
Let $k,\ell\in\bZ$ with $1\le k\le m<\ell\le n$.  Suppose that 
$\ua=(a_1,\dots,a_n)$ and $\ub=(b_1,\dots,b_n)$ are distinct points of $\bR^n$ 
which satisfy
\begin{equation}
 \label{type:prop:converse:eq}
 \left\{
 \begin{aligned}
 &a_1<\cdots<a_m=a_{m+1}<\cdots<a_n
 \et
 b_1<\cdots<b_m=b_{m+1}<\cdots<b_n, \\
 &a_j=b_j \quad\text{for $j=1,\dots,k-1$ and $j=\ell+1,\dots,n$,}\\
 &a_j\le b_{j-1} \quad\text{for $j=k+1,\dots,\ell$.}
 \end{aligned}
 \right.
\end{equation}
Set $a=a_1+\cdots+a_n$ and $b=b_1+\cdots+b_n$.  Then, we have $a<b$ 
and there exists a unique non-degenerate sign-free-$n$-system 
$\uP=(P_1,\dots,P_n)$ on $[a,b]$ with the following properties
\begin{itemize}
 \item[(i)] $\uP(a)=\ua$ and $\uP(b)=\ub$,
 \smallskip
 \item[(ii)] $P_m(q)<P_{m+1}(q)$ for each $q\in(a,b)$,
 \smallskip
 \item[(iii)] for each $j=k,\dots,\ell-1$ with $j\neq m$, there is exactly one value 
  of $q$ in $(a,b)$ for which $P_j(q)=P_{j+1}(q)$.
\end{itemize}
For this map, the sum $P_1+\cdots+P_m$ has slope $0$ then $1$ on $[a,b]$.
\end{proposition}

\begin{figure}[h]
     \begin{tikzpicture}[scale=0.75]
      \draw[dashed] (1,7) -- (1,-0.5) node[below] {$r_0=a$}; 
      \draw[dashed] (2.2,5.2) -- (2.2,0.3) node[below] {$r_1$};
      \node[below] at (3.1,0.7){$\dots$};
      \draw[dashed] (4,7) -- (4,0.3) node[below] {\ \ $r_{\ell-m-1}$};
      \draw[dashed] (6,9) -- (6,-0.5) node[below] {$t_{\ell-m-1}$};
      \draw[dashed] (6.8,7.8) -- (6.8,0.3) node[below]{\ \ $t_{\ell-m-2}$};
      \node[below] at (7.4,0.7){$\dots$};
      \draw[dashed] (8,5.2) -- (8,0.3) node[below]{\ \ $t_1$};
      \draw[dashed] (8.8,9) -- (8.8,-0.5) node[below] {$t_0$}; 
      \draw[dashed] (12,3) -- (12,1) node[below] {}; 
      \draw[dashed] (17,9) -- (17,-0.5) node[below] {$b$};
      \node[left] at (8,6.4){$\ddots$};
      \node[left] at (10.8,3.6){$\ddots$};
      \node[left] at (1,4){$a_m=a_{m+1}$};
      \draw[thick] (1,4) -- (6,9) -- (17,9) node[right]{$b_\ell$};
      \draw[thick] (1,4) -- (8.8,4) -- (9.6,4.8) -- (17,4.8) node[right]{$b_{m-1}$};
      \node[left] at (1,2){$a_{k+1}$};
      \draw[thick] (1,2.2) -- (11.2,2.2) -- (12,3) -- (17,3) node[right]{$b_k$};
      \node[left] at (1,1){$a_k$};
      \draw[thick] (1,1) -- (12,1) -- (17,6);
      \node[left] at (1,5.2){$a_{m+2}$};
      \draw[thick] (1,5.2) -- (8,5.2) -- (8.8,6) -- (17,6) node[right]{$b_m=b_{m+1}$};
      \node[left] at (1,7){$a_\ell$};
      \draw[thick] (1,7) -- (6,7) -- (6.8,7.8) -- (17,7.8) node[right]{$b_{\ell-1}$};
      \node[above] at (13,9) {$P_\ell$};
      \node[above] at (13,7.8) {$P_{\ell-1}$};
      \node[below] at (13,7.8) {$\vdots$};
      \node[above] at (13,6) {$P_{m+1}$};
      \node[above] at (13,4.8) {$P_m$};
      \node[below] at (13,4.8) {$\vdots$};
      \node[above] at (13,3) {$P_{k+1}$};
      \node[below] at (13,1.7) {$P_k$};
      \node[left] at (0.8,6.3){$\vdots$};
      \node[left] at (0.8,3.2){$\vdots$};
      \node[right] at (17.2,7.1){$\vdots$};
      \node[right] at (17.2,4.1){$\vdots$};
   \end{tikzpicture}
\caption{The combined graph of $(P_k,\dots,P_\ell)$ in Proposition \ref{type:prop:converse}.}
\label{type:fig1}
\end{figure}
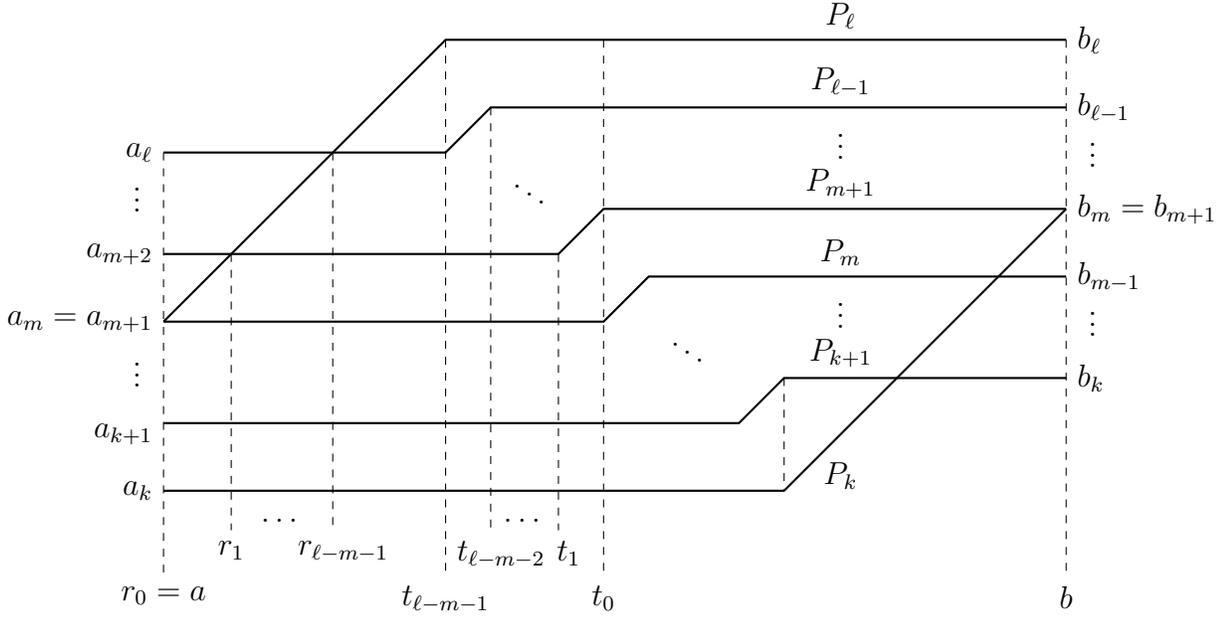

\begin{proof} 
\textbf{Existence.}
Using the geometric charaterization given in Section \ref{ssec:pgn:char}, it is a simple
matter to check that Figure \ref{type:fig1} represents the combined graph of a map 
$\uP=(P_1,\dots,P_n)$ with the required properties, provided that we supplement this graph
above and below by horizontal line segments with ordinate $a_j=b_j$ for each $j$ with
$j<k$ or $j>\ell$, if any.  Note that this drawing simplifies when $k=m$ or $\ell=m+1$.

\textbf{Uniqueness.} 
Suppose that a map $\uP=(P_1,\dots,P_n)$ has the required properties.  To prove that $\uP$ 
is the same as above, it suffices to show that the sums $S^-_j=P_1+\cdots+P_j$ are the 
same for each $j=1,\dots,n$.  We will do it only for $S^-_m,\dots,S^-_n$ as the 
argument is similar for $S^-_1,\dots,S^-_m$.  In fact, since $P_j$ is constant equal
to $a_j=b_j$ for each $j>\ell$, the sums $S^-_j$ with $j\ge \ell$ are affine of slope $1$ on 
$[a,b]$, given by  
\[
 S^-_j(q) = q-a+(a_1+\cdots+a_j).
\]
So, it suffices only to consider $S^-_m,\dots,S^-_{\ell-1}$.   
For each $i=0,\dots,\ell-m-1$, we need to show that, upon setting
\[
 r_i=a+a_{m+i+1}-a_{m+1} 
 \et
 t_i=(a_1+\cdots+a_{m+i+1})+(b_{m+i+1}+\cdots+b_n)-a_{m+1},
\]
the function $S^-_{m+i}$ is affine of slope $1$ on $[a,r_i]$, constant on $[r_i,t_i]$, and 
affine of slope $1$ on $[t_i,b]$.  This characterizes uniquely $S^-_{m+i}$ since 
$S^-_{m+i}(a)=a_1+\cdots+a_{m+i}$.    

By (ii), the interval $[a,b]$ is simple for $\uP$.  Then Lemma \ref{part:lemma:simple} shows 
that $S_m$ is constant on $[a,t_0]$ and affine of slope $1$ on $[t_0,b]$.  Since $r_0=a$, 
this proves our claim for $i=0$.  Now, suppose by induction that $S^-_{m+i}$ has the 
stated property for some $i$ with $0\le i\le\ell-m-2$.  Since $m < m+i+1 < \ell$, 
condition (iii) tells us that there is exactly one $r\in(a,b)$ for which 
$P_{m+i+1}(r)=P_{m+i+2}(r)$.  Since $\uP$ is non-degenerate,
it follows from Lemma \ref{part:lemma:Am} that $r$ is a point of $(a,b)$ at 
which $S^-_{m+i+1}$ changes slope from $1$ to $0$, and the only such point.  
Moreover, $P_{m+1}$ has slope $1$ to the right of $a$, because 
$P_m(a)=P_{m+1}(a)$.  Thus, $S^-_{m+i+1}$ also has slope $1$ 
to the right of $a$.    Consequently, it has slope $1$ on $[a,r]$, and then slope $0$ 
on $[r,t]$ and slope $1$ on $[t,b]$ for some $t\in (r,b]$.  It remains only to prove
that $r=r_{i+1}$ and $t=t_{i+1}$.

To prove this, we first observe that, since $S^-_{m+i+1}$ has slope $1$ on $[a,r]$,
then $P_{m+i+2}$ is constant equal to $a_{m+i+2}$ on $[a,r]$, thus
$P_{m+i+1}(r)=P_{m+i+2}(r)=a_{m+i+2}$.   Since $S^-_{m+i+1}$ is constant on $[r,t]$, 
both $S^-_{m+i}$ and $P_{m+i+1}$ are constant on $[r,t]$.  Finally, since $S^-_{m+i}$
has slope $1$ on $[a,r_i]$, is constant on $[r_i,t_i]$ and has slope $1$ on $[t_i,b]$, 
we conclude that $r_i\le r$, that $t\le t_i$, and that $P_{m+i+1}$ is constant
equal to $a_{m+i+1}$ on $[a,r_i]$, has slope $1$ on $[r_i,r]$, is constant equal to 
$a_{m+i+2}$ on $[r,t]$, has slope $1$ on $[t,t_i]$, and is constant equal to 
$b_{m+i+1}$ on $[t_i,b]$.  This yields
\[
 r=r_i+a_{m+i+2}-a_{m+i+1}=r_{i+1}
 \et
 t=t_i+a_{m+i+2}-b_{m+i+1}=t_{i+1},
\]
as claimed.
\end{proof}

\begin{corollary}
\label{type:cor1}
For each $i\ge 1$, let $\ua^{(i)}=(a^{(i)}_1,\dots,a^{(i)}_n)\in\bR^n$ with
\begin{equation}
 \label{type:cor1:eq1}
 0 \le a^{(i)}_1<\cdots<a^{(i)}_m=a^{(i)}_{m+1}<\cdots<a^{(i)}_n,
\end{equation}
and set $a_i=a^{(i)}_1+\cdots+a^{(i)}_n$.  Suppose that, for each $i\ge 1$, 
the points $\ua^{(i)}$ and $\ua^{(i+1)}$ are distinct and satisfy  
\begin{equation}
 \label{type:cor1:eq2} 
 \left\{
 \begin{aligned}
  &a^{(i)}_j=a^{(i+1)}_j &&\text{for $j=1,\dots,k_i-1$ and $j=\ell_i+1,\dots,n$,}\\
  &a^{(i)}_j \le a^{(i+1)}_{j-1} &&\text{for $j=k_i+1,\dots,\ell_i$,}
 \end{aligned}
 \right.
\end{equation}
for integers $k_i$ and $\ell_i$ with $1\le k_i\le m< \ell_i\le n$.
Finally suppose that $a^{(i)}_1$ tends to infinity with $i$.
Then there exists a non-degenerate proper $n$-system $\uR=(R_1,\dots,R_n)$ 
on $[a_1,\infty)$ such that
\begin{itemize}
 \item[(i)] $\uR(a_i)=\ua^{(i)}$ for each $i\ge 1$,
 \smallskip
 \item[(ii)] $\cA_m(\uR) = \{a_1,a_2,a_3,\dots\}$.
\end{itemize}
Moreover, if there exist an integer $s\ge 1$ and a number $\rho>1$ such that
$\ua^{(i+s)}=\rho\ua^{(i)}$ for each $i\ge 1$, we may choose $\uR$ to be 
self-similar with $\uR(\rho q)=\rho\uR(q)$ for each $q\ge a_1$.
\end{corollary}

\begin{proof}
For each $i\ge 1$, the points $\ua=\ua^{(i)}$ and $\ub=\ua^{(i+1)}$ satisfy
the conditions \eqref{type:prop:converse:eq} of the proposition for the choice of
$k=k_i$ and $\ell=\ell_i$.  The proposition thus provides a non-degenerate 
sign-free-$n$-system $\uP^{(i)}=(P^{(i)}_1,\dots,P^{(i)}_n)$ on $[a_i,a_{i+1}]$ 
such that 
\[
 \uP^{(i)}(a_i)=\ua^{(i)}, \quad
 \uP^{(i)}(a_{i+1})=\ua^{(i+1)}, \quad
 P^{(i)}_m(q)<P^{(i)}_{m+1}(q) \ \text{for each $q\in(a_i,a_{i+1})$.}
\]
Consider the map $\uR=(R_1,\dots,R_n)\colon[a_1,\infty)\to\bR^n$ which coincides 
with $\uP^{(i)}$ on $[a_i,a_{i+1}]$ for each $i\ge 1$.  This is a sign-free-$n$-system
on $[a_1,\infty)$ because, for each $i\ge 2$, there is exactly one $j\in\{1,\dots,n\}$ 
for which the sum $R_1+\dots+R_j$ changes slope from $1$ to $0$ at 
$a_i$, it is $j=m$, and we have $R_m(a_i)=R_{m+1}(a_i)$.  Since $R_1(a_i)=a^{(i)}_1\ge 0$
tends to infinity with $i$, this is a proper $n$-system.  By construction, it is
non-degenerate, with the required properties (i) and (ii).

Finally, suppose that, for each $i\ge 1$, we have $\ua^{(i+s)}=\rho\ua^{(i)}$ for 
some fixed integer $s\ge 1$ and some $\rho>1$.  Then, for each $i\ge 1$, we
have $k_{i+s}=k_i$, $\ell_{i+s}=\ell_i$, and $a_{i+s}=\rho a_i$, thus 
$[a_{i+s+1},a_{i+s}] = \rho [a_{i+1},a_i]$ and 
$\uP^{(i+s)}(q)=\rho\uP^{(i)}(q/\rho)$ for each $q\in [a_{i+s+1},a_{i+s}]$ by 
uniqueness of the construction in Proposition \ref{type:prop:converse}.  This yields 
$\uR(\rho q)=\rho\uR(q)$ for each $q\ge a_1$.
\end{proof}

\begin{remark}
\label{type:remark}
The maps constructed by the above process are simpler than general 
non-degenerate proper $n$-systems, and they suffice to compute the spectrum 
$\cS_{n-m,n}$.  

Indeed, choose any non-degenerate proper $n$-system $\uP=(P_1,\dots,P_n)$,  
and list the elements of $\cA_m(\uP)$ in increasing order $a_1<a_2<a_3<\dots$.
By Proposition \ref{type:prop:kl},  the points $\ua^{(i)}=\uP(a_i)$ with $i\ge 1$ 
satisfy the hypotheses of the corollary for a choice of integers $k_i$ and $\ell_i$.  
So, the corresponding non-degenerate 
proper $n$-system $\uR=(R_1,\dots,R_n)$ on $[a_1,\infty)$ has the same set of 
$m$-division numbers $\cA_m(\uR)=\{a_1,a_2,a_3,\dots\}$ as $\uP$,  and $\uR$ 
coincides with $\uP$ on that set.  By Proposition \ref{part:prop:n-sys} and 
Corollary~\ref{pgn:tool:cor}, we conclude that $\uP$ and $\uR$ yield the same point
\[
 (\chibot_m(\uP),\chitop_m(\uP)) = (\chibot_m(\uR),\chitop_m(\uR))
\]
in $\cS_{n-m,n}$.  Since each interval $[a_i,a_{i+1}]$ is simple for $\uP$ and for $\uR$, 
Lemma~\ref{part:lemma:simple} further shows that the sums $P_1+\cdots+P_m$  
and $R_1+\dots+R_m$ coincide on $[a_i,a_{i+1}]$ for each $i\ge 1$, and so they
coincide on $[a_1,\infty)$. 
\end{remark}

%
%

\section{Constructions}
\label{sec:cons}

Let $m,n\in \bZ$ with $1\le m\le n-1$.  We recall that the spectrum $\cS_{n-m,n}$ 
of the pair $(\homega_{n-m-1},\omega_{n-m-1})$ is the same relative to any
embedding of a number field $K$ into its completion $K_\pw$ at a place $\pw$ of 
$K$, and that Corollary~\ref{pgn:tool:cor} allows us to compute it in terms 
of $n$-systems.  In this section, we give an alternative description of that spectrum 
in the spirit of the remark ending the preceding section.  Then, we use it to prove
Proposition~\ref{dr:prop:Ri} and Theorem~\ref{dr:thm:rectangles}.

\begin{theorem}
 \label{cons:thm}
Suppose that, for $i=1,\dots,s+1$ with $s\ge 1$, a point 
$\ua^{(i)}=(a^{(i)}_1,\dots,a^{(i)}_n)$ of $\bR^n$ satisfies the hypotheses
\eqref{type:cor1:eq1} of Corollary~\ref{type:cor1}.  Suppose moreover that,
for $i=1,\dots,s$, the points $\ua^{(i)}$ and $\ua^{(i+1)}$ are distinct and
satisfy \eqref{type:cor1:eq2} for a choice of integers
$k_i$, $\ell_i$ with $1\le k_i\le m<\ell_i\le n$.  Finally, suppose that 
$\ua^{(s+1)}=\rho\ua^{(1)}$ for some $\rho>1$, and that all coordinates
of $\ua^{(1)}$ are positive.  Then $\cS_{n-m,n}$ contains the point 
$(\alpha,\beta)$ given by
\begin{equation}
 \label{cons:thm:eq1}
 \alpha=\min_{1\le i\le s}\frac{A^+_i}{A^-_i},
 \quad
 \beta=\max_{1\le i\le s}\frac{A^+_{i+1}}{A^-_i},
\end{equation}
where
\begin{equation}
 \label{cons:thm:eq2}
 A^-_i = a^{(i)}_1+\cdots+a^{(i)}_m,
 \quad
 A^+_i = a^{(i)}_{m+1}+\cdots+a^{(i)}_n,
\end{equation}
for each $i=1,\dots,s+1$.  More precisely, $\cS_{n-m,n}$ is the topological 
closure in $[(n-m)/m,\infty]^2$ of the set of points $(\alpha,\beta)$ 
constructed as above.
\end{theorem}

\begin{proof}
We claim that the points $(\alpha,\beta)$ constructed by the theorem are the 
same as the pairs $(\chibot_m(\uR),\chitop_m(\uR))$ where $\uR$ runs through 
all non-degenerate self-similar $n$-systems.  If we admit this, the 
result follows from Corollary~\ref{pgn:tool:cor} together with the lower bound
$\chibot_m(\uR)\ge (n-m)/m$ coming from \eqref{part:prop:n-sys:eq} in 
Proposition~\ref{part:prop:n-sys}.

In one direction, suppose $(\ua^{(1)},\dots,\ua^{(s+1)})$ as in the theorem.
We extend it to an infinite sequence $(\ua^{(i)})_{i\ge 1}$ by defining 
recursively $\ua^{(i+s)}=\rho\ua^{(i)}$ for each $i\ge 1$.  We also extend 
$(k_1,\dots,k_s)$ and $(\ell_1,\dots,\ell_s)$ to infinite $s$-periodic sequences 
$(k_i)_{i\ge 1}$ and $(\ell_i)_{i\ge 1}$ respectively.  Then the hypotheses of 
Corollary \ref{type:cor1} are fulfilled and so, upon denoting by $a_i$ the sum 
of the coordinates of $\ua^{(i)}$ for each $i\ge 1$, there is a non-degenerate 
self-similar $n$-system $\uR=(R_1,\dots,R_n)$ on $[a_1,\infty)$ such that 
$\cA_m(\uR)=\{a_1,a_2,a_3,\dots\}$ and $\uR(a_i)=\ua^{(i)}$ for 
each $i\ge 1$.  By Proposition~\ref{part:prop:n-sys}, we have 
\begin{equation*}
 \label{cons:thm:eq3}
 \chibot_m(\uR)=\liminf_{i\to\infty}\frac{A^+_i}{A^-_i}
 \et
 \chitop_m(\uR)=\limsup_{i\to\infty}\frac{A^+_{i+1}}{A^-_i}
\end{equation*}
where $A^-_i$ and $A^+_i$ are given by \eqref{cons:thm:eq2} for each 
$i\ge 1$.  Since $\ua^{(i+s)}=\rho\ua^{(i)}$ for each $i\ge 1$, the 
ratios $A^+_i/A^-_i$ and $A^+_{i+1}/A^-_i$ are periodic of period $s$.
Thus, these formulas simplify to $\chibot_m(\uR)=\alpha$ and 
$\chitop_m(\uR)=\beta$ where $\alpha$ and $\beta$ are given by
\eqref{cons:thm:eq1}.

Conversely, suppose that $\uR\colon[q_0,\infty)\to\bR^n$ is any 
non-degenerate self-similar (thus proper) $n$-system.  Let $a_1<a_2<a_3<\cdots$ be the 
elements of $\cA_m(\uR)$ listed in increasing order, and set $\ua^{(i)}=\uR(a_i)$
for each $i\ge 1$.  Since $\uR$ is self-similar, there exists $\rho>1$ such
that $\uR(\rho q)=\rho\uR(q)$ for each $q\ge q_0$, and so $\rho\cA_m(\uR)
=\cA_m(\uR)\cap[\rho q_0,\infty)$.  Consequently, there exists an 
integer $s\ge 1$ such that $a_{i+s}=\rho a_i$ and $\ua^{(i+s)}=\rho\ua^{(i)}$ 
for each $i\ge 1$.  By Propositions \ref{type:prop:kl} and \ref{part:prop:n-sys},
the finite subsequence $(\ua^{(1)},\dots,\ua^{(s+1)})$ fulfills the hypotheses 
of the theorem and yields the point 
$(\alpha,\beta)=(\chibot_m(\uR),\chitop_m(\uR))$ through
\eqref{cons:thm:eq1}. 
\end{proof}

\begin{proof}[\textbf{Proof of Proposition \ref{dr:prop:Ri}}] 
Let $g,\rho\in\bR$ with $1< g\le \rho$.  The points 
\[
  \ua^{(1)}=(1,\dots,g^{m-1},g^{m-1},\dots,g^{n-2}) 
  \et
  \ua^{(2)}=\rho \ua^{(1)} 
\]
satisfy the hypotheses of Theorem~\ref{cons:thm} with $s=1$, $k_1=1$ 
and $\ell_1=n$.  Thus $\cS_{n-m,n}$ contains the point $(\alpha,\beta)$
where
\begin{align}
 \alpha
  &=\frac{g^{m-1}+\cdots+g^{n-2}}{1+\cdots+g^{m-1}}
    = \Big(\sum_{i=0}^{n-m-1}g^i\Big) / \Big(\sum_{i=0}^{m-1}1/g^i\Big),
  \label{cons:proof:eq}\\
  \beta
  &=\frac{\rho(g^{m-1}+\cdots+g^{n-2})}{1+\cdots+g^{m-1}}
    =\rho\alpha.
 \notag
\end{align}
As $\cS_{n-m,n}$ is a closed subset of $[0,\infty]^2$, it thus contains 
all $(\alpha,\beta)\in[0,\infty]^2$ which, for some $g\in[1,\infty]$, satisfy 
\eqref{cons:proof:eq} and  $\beta\ge g\alpha$.  Proposition \ref{dr:prop:Ri}
follows upon replacing $m$ by $n-m$.
\end{proof}

The following complement to Theorem~\ref{cons:thm} is inspired by
\cite[Proposition 3.4]{R2016} and is key to the  proof of 
Theorem~\ref{dr:thm:rectangles}.

\begin{proposition}
\label{cons:prop}
Let the notation and hypotheses be as in Theorem~\ref{cons:thm}, and choose
a (strictly) increasing function $\nu\colon[0,\infty)\to[0,\infty)$.  Define 
\[
 \tua^{(i)} = \big(\nu(a^{(i)}_1),\dots,\nu(a^{(i)}_n)\big)
 \quad\text{for $i=1,\dots,s+1$,}
\]
and suppose that $\tua^{(s+1)}=\trho\/\tua^{(1)}$ for some $\trho>1$.  Then 
$\cS_{n-m,n}$ contains the point $(\talpha,\tbeta)$ given by
\begin{equation}
 \label{cons:prop:eq1}
 \talpha=\min_{1\le i\le s} \tA^+_i/\tA^-_i,
 \quad
 \tbeta=\max_{1\le i\le s} \tA^+_{i+1}/\tA^-_i,
\end{equation}
where
\begin{equation}
 \label{cons:prop:eq2}
 \tA^-_i = \nu(a^{(i)}_1)+\cdots+\nu(a^{(i)}_m),
 \quad
 \tA^+_i = \nu(a^{(i)}_{m+1})+\cdots+\nu(a^{(i)}_n),
\end{equation}
for each $i=1,\dots,s+1$.  Moreover,  
\begin{itemize}
\item[(i)] if $\nu(t)/t$ is non-decreasing on $(0,\infty)$, then $\talpha\ge \alpha$ and $\tbeta\ge \beta$;
\smallskip
\item[(ii)] if $\nu(t)/t$ is non-increasing on $(0,\infty)$, then $\talpha\le \alpha$ and $\tbeta\le \beta$.
\end{itemize}
\end{proposition}

\begin{proof}
The first assertion follows immediately from Theorem~\ref{cons:thm} since 
$\tua^{(1)},\dots,\tua^{(s+1)}$ satisfy the hypotheses of that theorem for the same 
integers $k_i$ and $\ell_i$ with $1\le i\le s$ as the given points 
$\ua^{(1)},\dots,\ua^{(s+1)}$.  

Now, suppose that $\nu(t)/t$ is non-decreasing on $(0,\infty)$.  For each 
$i=1,\dots, s+1$, we find that 
\[
 \tA_i^- = \sum_{j=1}^m \frac{\nu(a^{(i)}_j)}{a^{(i)}_j}a^{(i)}_j
   \le \frac{\nu(a^{(i)}_m)}{a^{(i)}_m} A_i^-\,,
   \quad
 \tA_i^+ = \sum_{j=m+1}^n \frac{\nu(a^{(i)}_j)}{a^{(i)}_j}a^{(i)}_j
   \ge \frac{\nu(a^{(i)}_m)}{a^{(i)}_m} A_i^+,
\]
thus $\tA_i^+/\tA_i^-\ge A_i^+/A_i^-$, and so $\talpha\ge \alpha$.  
For $i=1,\dots,s$, we also have $a^{(i)}_m\le a^{(i+1)}_m$, and the 
above yields $\tA_{i+1}^+/\tA_i^-\ge A_{i+1}^+/A_i^-$, thus 
$\tbeta\ge \beta$.   This proves the additional assertion (i).  The proof 
of (ii) is similar.
\end{proof}

For a fixed $c>0$, the above proposition applies with 
$\nu(t)=t^c$ for each $t\ge 0$, independently of the choice of
$\ua^{(1)},\dots,\ua^{(s+1)}$, for then we have $\tua^{(s+1)}
=\rho^c\/\tua^{(1)}$.  However, in the application below, the function 
$\nu$ is tailored to the choice of the sequence.

\begin{proof}[Proof of Theorem \ref{dr:thm:rectangles}]
Let $\cS$ denote the subset of $\cS_{n-m,n}$ constructed by
Theorem~\ref{cons:thm}.  Since, by this theorem, $\cS_{n-m,n}$ is the topological
closure of $\cS$ in $[(n-m)/m,\infty]^2$, it suffices to show that $\cS$ 
contains both the vertical half-line $\{\alpha\}\times[\beta,\infty)$ and 
the possibly empty horizontal line segment $((n-m)/m,\alpha]\times\{\beta\}$ 
for each $(\alpha,\beta)\in\cS$.

To this end, we fix a choice of $(\alpha,\beta)\in\cS$ coming from a
sequence $\big(\ua^{(1)},\dots,\ua^{(s+1)}\big)$ with 
$\ua^{(s+1)}=\rho\ua^{(1)}$, as in Theorem~\ref{cons:thm}.   Then, 
with the notation of the theorem, we choose an integer $r\ge 1$ such that 
\begin{equation}
 \label{cons:prop:eq:r}
 a_n^{(1)} < \rho^r a_1^{(1)},
\end{equation}
set $\ts=(2r+1)s$, and form the extended sequence 
$\big(\ua^{(1)},\dots,\ua^{(\ts+1)}\big)$
where $\ua^{(i+s)}=\rho\ua^{(i)}$ for $i=1,\dots,2rs+1$. 
This new sequence again satisfies the hypotheses of Theorem~\ref{cons:thm}
and yields the same point $(\alpha,\beta)$ through the formulas 
\eqref{cons:thm:eq1} and \eqref{cons:thm:eq2}, with $s$ replaced by $\ts$, 
because the ratios $A^+_i/A^-_i$ and $A^+_{i+1}/A^-_i$ are 
periodic functions of $i$ with period $s$ for $1\le i\le \ts$.  In particular,
we note that
\begin{equation}
 \label{cons:prop:eq:ab}
 \alpha=\min_{2rs+1\le i\le \ts} A^+_i/A^-_i
 \et
 \beta=\max_{2rs+1\le i\le \ts} A^+_{i+1}/A^-_i
\end{equation}
Moreover, the choice of $r$ in \eqref{cons:prop:eq:r} ensures that 
\begin{equation}
 \label{cons:prop:eq:aa}
 a_n^{(1)} < a_1^{(rs+1)} \le a_n^{(rs+1)} < a_1^{(2rs+1)}.
\end{equation}
Below, we apply Proposition~\ref{cons:prop} to the sequence 
$\big(\ua^{(1)},\dots,\ua^{(\ts+1)}\big)$ with various increasing 
functions $\nu$, each of which satisfies
\begin{equation}
 \label{cons:prop:eq:phi}
 \begin{cases} 
  \nu(t)=t &\text{for $0\le t\le a_n^{(1)}$,} \cr 
  \nu(t)=ct &\text{for $t\ge a_1^{(2rs+1)}$,}
 \end{cases}
\end{equation}
with a constant $c>0$.  For the corresponding sequence 
$\big(\tua^{(1)},\dots,\tua^{(\ts+1)}\big)$, this yields 
\[
 \tua^{(1)} = \ua^{(1)} \et \tua^{(i)} = c\ua^{(i)} \quad \text{for \ $2rs+1\le i\le \ts+1$.}
\]
In particular, we have $\tua^{(\ts+1)}=\trho\tua^{(1)}$ with $\trho=c\rho^{2r+1}$.
We also note that $\tA^\pm_i=cA^\pm_i$ for each choice of sign $\pm$ and 
each $i=2rs+1,\dots,\ts+1$.  By \eqref{cons:prop:eq:ab}, this implies that the 
corresponding point $(\talpha,\tbeta)\in\cS$, given by \eqref{cons:prop:eq1}, 
satisfies 
\begin{equation}
 \label{cons:prop:eq:tab}
 \alpha=\min_{2rs+1\le i\le \ts} \tA^+_i/\tA^-_i \ge \talpha
 \et
 \beta=\max_{2rs+1\le i\le \ts} \tA^+_{i+1}/\tA^-_i \le \tbeta.
\end{equation}

We first consider the increasing function $\nu\colon[0,\infty)\to[0,\infty)$ 
given by
\[
 \nu(t)=\begin{cases} 
      t &\text{if $0\le t\le a_n^{(1)}$,}\cr
    ct &\text{if $t> a_n^{(1)}$,}
    \end{cases}
\]
for a choice of $c\ge 1$.  It fulfills 
conditions \eqref{cons:prop:eq:phi} because of \eqref{cons:prop:eq:aa}.
Thus, \eqref{cons:prop:eq:tab} holds.   As $\nu(t)/t$ is non-decreasing 
on $(0,\infty)$, Proposition~\ref{cons:prop} (i) also gives $\talpha\ge \alpha$.
Thus, we have $\talpha=\alpha$.   Moreover, by \eqref{cons:prop:eq:aa},
there exists an integer $\ell$ with $1\le \ell\le rs$ such that
$a_n^{(1)}=a_n^{(\ell)}<a_n^{(\ell+1)}$.   For that choice of $\ell$, we
have $\tA^-_\ell=A^-_\ell$ and $\tA^+_{\ell+1}\ge c a_n^{(\ell+1)}$,
thus $\tbeta\ge c a_n^{(\ell+1)}/A^-_\ell$ tends to infinity with $c$.
As $\tbeta$ is a continuous function of $c$ with $\tbeta=\beta$ for $c=1$,
it takes all values in $[\beta,\infty)$ as $c$ varies in $[1,\infty)$.  So, 
$\cS$ contains $\{\alpha\}\times [\beta,\infty)$.

Finally, choose $\epsilon\in(0,1]$, set $\ell=rs+1$, and let 
$\nu\colon[0,\infty)\to[0,\infty)$ be the continuous function given by
\[
 \nu(t)=\begin{cases} 
      t &\text{if $0\le t\le a_1^{(\ell)}$,}\cr
    (1-\epsilon)a_1^{(\ell)}+\epsilon t 
          &\text{if $a_1^{(\ell)}\le t \le  a_n^{(\ell)}$,}\cr
     ct &\text{if $a_n^{(\ell)}\le t$,}
    \end{cases}
    \quad\text{where}\quad c=(1-\epsilon)\frac{a_1^{(\ell)}}{a_n^{(\ell)}}+\epsilon.
\]
It is increasing and fulfills \eqref{cons:prop:eq:phi} in view of 
\eqref{cons:prop:eq:aa}.  Thus, \eqref{cons:prop:eq:tab} holds.   As 
$\nu(t)/t$ is non-increasing on $(0,\infty)$, Proposition~\ref{cons:prop} (ii) 
further gives $\tbeta\le \beta$.  So, we have $\tbeta=\beta$.  We also find that
\[
 \talpha \le \frac{\tA_\ell^+}{\tA_\ell^-}
   =\frac{(1-\epsilon)(n-m)a_1^{(\ell)}+\epsilon A^+_\ell}%
             {(1-\epsilon)ma_1^{(\ell)}+\epsilon A^-_\ell}.
\]
Since $\talpha$ is a continuous function of $\epsilon\in(0,1]$
that is bounded from below by $(n-m)/m$, the above estimate shows 
that $\talpha$ tends
to $(n-m)/m$ as $\epsilon\to 0$.  As $\talpha=\alpha$ for $\epsilon=1$, this 
means that $\talpha$ takes all values in $((n-m)/m,\alpha]$ as $\epsilon$ varies 
in $(0,1]$ and so $\cS$ contains $((n-m)/m,\alpha]\times\{\beta\}$.
\end{proof}

%
%

\section{Proof of Theorem \ref{dr:thm:MM}}
\label{sec:MM}

Theorem~\ref{dr:thm:MM} describes the spectra $\cS_{1,n}$ and $\cS_{n-1,n}$ 
for each integer $n\ge 2$.  For $n=2$, the result is easy.  
By Corollary~\ref{pgn:tool:cor}, the set $\cS_{1,2}$ consists of the points
$(\chibot_1(\uP),\chitop_1(\uP))$ where $\uP$ is a proper non-degenerate 
$2$-system, while Proposition~\ref{part:prop:n-sys} shows that these points 
satisfy $1=\chibot_1(\uP)\le \chitop_1(\uP)$.  Thus we have 
$\cS_{1,2}\subseteq\{1\}\times[1,\infty]$.  As Proposition~\ref{dr:prop:Ri}
shows the reverse inclusion, we conclude that $\cS_{1,2}=\{1\}\times[1,\infty]$.  
Alternatively, this can be deduced from Theorem~\ref{cons:thm}.  Thus, for the proof
of Theorem~\ref{dr:thm:MM}, we may assume that $n\ge 3$.  

The main step is provided by the following result which encompasses 
\cite[Propositions 2.3.2 and 2.3.6]{Ri2019}).   We refer to Definition~\ref{part:def} 
for the meaning of the symbols $\cA_1(\uP)$ and $\cI_1(\uP)$ used below.   

\begin{lemma}
\label{MM:lemma}
Let $\uP=(P_1,\dots,P_n)\colon I\to\bR^n$ be a non-degenerate
sign-free-$n$-system for some $n\ge 3$.  Suppose that a compact sub-interval 
$J$ of $I$ contains at least $n-2$ simple intervals of type $\{1,\dots,n\}$ 
from $\cI_1(\uP)$.   Then, there are numbers 
\begin{equation}
 \label{MM:lemma:eq1}
  a_1<b_1 \le a_2<b_2 \le \cdots \le a_{n-2}<b_{n-2}
\end{equation}
in $J\cap \cA_1(\uP)$ which, for each $i=1,\dots,n-2$, satisfy the following properties:
\begin{itemize}
\item[(P1)] $[a_i,b_i]$ is a simple interval in $\cI_1(\uP)$ whose type contains $\{1,\dots,n-i+1\}$;
 \smallskip
\item[(P2)] if $i>1$, the type of $[b_{i-1},a_{i}]$ is contained in $\{1,\dots,n-i\}$.
\end{itemize}
For such a choice of numbers, we have
\begin{equation}
 \label{MM:lemma:eq2}
 (A_1^+ - A_1^-) + \cdots +  (A_{n-2}^+ - A_{n-2}^-) 
 \le B_1^+ + \cdots + B_{n-3}^+ + B_{n-2}^-
\end{equation} 
where, for each $i=1,\dots,n-2$, 
\begin{equation}
 \label{MM:lemma:eq3}
 \begin{aligned}
   A_i^-&=P_1(a_i), &&A_i^+=P_2(a_i)+\cdots+P_n(a_i), \\
   B_i^-&=P_1(b_i), &&B_i^+=P_2(b_i)+\cdots+P_n(b_i).
 \end{aligned} 
\end{equation}
\end{lemma}

\begin{proof}
Consider the elements $c_1<c_2<\cdots<c_s$ of $J \cap \cA_1(\uP)$.  
For each $i=1,\dots,s-1$, the interval $[c_i,c_{i+1}]$ is a simple 
interval in $\cI_1(\uP)$ and its type is $\{1,\dots,\ell_i\}$ for some integer 
$\ell_i$ with $2\le \ell_i\le n$.  By hypothesis, we have $\ell_i=n$ for at least 
$n-2$ values of $i$.  Set $i(0)=0$.  Then, for $j=1,\dots,n-2$, we may define 
recursively $i(j)$ to be the smallest integer with $i(j-1)<i(j)<s$ and 
$\ell_{i(j)}\ge n-j+1$.   Then (P1) and (P2) are fulfilled by relabelling in the 
form \eqref{MM:lemma:eq1} the numbers
\[
 c_{i(1)} < c_{i(1)+1}  \le c_{i(2)} < c_{i(2)+1} \le \cdots \le c_{i(n-2)} < c_{i(n-2)+1}.
\]

To prove \eqref{MM:lemma:eq2}, we start by fixing an index $i$ with $1\le i\le n-2$, 
and write
\[
 \uP(a_i)=(a^{(i)}_1,\dots,a^{(i)}_n)
 \et
 \uP(b_i)=(b^{(i)}_1,\dots,b^{(i)}_n).
\]
Since the type of $[a_i,b_i]$ contains $\{1,\dots,n-i+1\}$, we have
\[
 a^{(i)}_j \le b^{(i)}_{j-1} \quad\text{for $j=2,\dots,n-i+1$.}
\]
Since $a^{(i)}_1=a^{(i)}_2$, we deduce that
\[
 \begin{aligned}
 A_i^+ - A_i^- 
  &= a^{(i)}_3+\cdots+a^{(i)}_n\\
  &\le (b^{(i)}_2+\cdots+b^{(i)}_{n-i})+(a^{(i)}_{n-i+2}+\cdots+a^{(i)}_n)\\
  &= (a^{(i)}_{n-i+2}+\cdots+a^{(i)}_n) + B_i^+ -  (b^{(i)}_{n-i+1}+\cdots+b^{(i)}_n).
 \end{aligned}
\]
If $i=1$, this can be rewritten in the form 
\begin{equation}
\label{MM:lemma:eq4}
  A_1^+ - A_1^- - B_1^+ \le - b^{(1)}_{n}.
\end{equation}
If $i>1$, the type of $[b_{i-1},a_i]$ is contained in $\{1,\dots,n-i\}$, so 
\[
 b^{(i-1)}_j=a^{(i)}_j  \quad\text{for $j=n-i+1,\dots,n$,}
\]
and we conclude that
\begin{equation}
\label{MM:lemma:eq5}
  A_i^+ - A_i^- -  B_i^+ 
  \le 
  (b^{(i-1)}_{n-i+2}+\cdots+b^{(i-1)}_{n}) - (b^{(i)}_{n-i+1}+\cdots+b^{(i)}_{n}).
\end{equation}
Summing term by term the inequalities \eqref{MM:lemma:eq4} and  
\eqref{MM:lemma:eq5} for $2\le i\le n-2$, we obtain
\[
 \sum_{i=1}^{n-2} ( A_i^+ - A_i^- -  B_i^+)
 \le - (b^{(n-2)}_3+\cdots+b^{(n-2)}_n) = B_{n-2}^- - B_{n-2}^+,
\]
since $b^{(n-2)}_{1}=b^{(n-2)}_{2}$, and \eqref{MM:lemma:eq2} follows.
\end{proof}

\begin{proposition}
\label{MM:prop:n-sys}
Let $\uP=(P_1,\dots,P_n)\colon [q_0,\infty)\to\bR^n$ be a proper 
non-degenerate $n$-system for some $n\ge 3$.  Suppose that 
$\chitop_1(\uP)<\infty$. Then, we have
\[
 n-1 \le \chibot_1(\uP)
      \le \sum_{i=0}^{n-2}\left(\frac{\chitop_1(\uP)}{ \chibot_1(\uP)}\right)^i.
\]
\end{proposition}

\begin{proof}
We already know that $n-1 \le \chibot_1(\uP)$ by Proposition \ref{part:prop:n-sys}.
Let $\alpha,\beta\in(0,\infty)$ with 
\[
 \alpha < \chibot_1(\uP) \le \chitop_1(\uP) < \beta.
\]  
It remains to show that
\begin{equation}
\label{MM:prop:n-sys:eq1}
 \alpha \le 1+g+\cdots+g^{n-2} \quad \text{where} \quad g=\beta/\alpha.
\end{equation}
To this end, we may assume that 
\begin{equation}
\label{MM:prop:n-sys:eq2}
 \alpha > 1+g+\cdots+g^{n-3} \ge 1.
\end{equation}

Let $c_1<c_2<c_3<\cdots$ denote the elements of $\cA_1(\uP)$.  
In view of formulas \eqref{part:prop:n-sys:eq} from 
Proposition~\ref{part:prop:n-sys}, there is an integer $r\ge 1$ such that, 
for each $i\ge r$, we have $P_1(c_i)>0$ and 
\[
 \alpha \le \frac{P_2(c_i)+\cdots+P_n(c_i)}{P_1(c_i)} 
           \le \frac{P_2(c_{i+1})+\cdots+P_n(c_{i+1})}{P_1(c_i)} \le \beta.
\]  
Moreover, since $\uP$ is proper, its component $P_n$ is unbounded.  Thus, there 
are infinitely many integers $i\ge r$ for which $P_n(c_i)<P_n(c_{i+1})$.  For each of 
these, the simple interval $[c_i,c_{i+1}]$ in $\cI_1(\uP)$ has type $\{1,\dots,n\}$.  
Thus, there is an integer $s>r$ such that $[c_r,c_s]$ contains at least $n-2$
simple intervals of type $\{1,\dots,n\}$ from $\cI_1(\uP)$.

We apply Lemma \ref{MM:lemma} to the subinterval $J=[c_r,c_s]$ of $[q_0,\infty)$.
By the choice of $r$, the corresponding numbers in \eqref{MM:lemma:eq3} 
are positive and so they satisfy
\begin{equation}
\label{MM:prop:n-sys:eq3}
 A^-_i\le\frac{1}{\alpha}A^+_i, \quad
 B^-_i\le\frac{1}{\alpha}B^+_i \et
 B^+_i\le \beta A^-_i \le g A^+_i
 \quad (1\le i  \le n-2).
\end{equation}
Using the above upper bounds for $A^-_i$ and $B^-_i$, we deduce 
from \eqref{MM:lemma:eq2} that
\begin{equation}
\label{MM:prop:n-sys:eq4}
 \Big(1-\frac{1}{\alpha}\Big)( A_1^+ + \cdots +  A_{n-2}^+ ) 
 \le B_1^+ + \cdots + B_{n-3}^+ + \frac{1}{\alpha}B_{n-2}^+ \,  .
\end{equation}
We also note that 
\begin{equation}
\label{MM:prop:n-sys:eq5}
 0 < B_{i-1}^+ \le A^+_i \quad (2\le i \le n-2)
\end{equation}
because of \eqref{MM:lemma:eq1} and the fact that $P_2+\cdots+P_n$ 
is a non-decreasing function.

For each $i=1,\dots,n-2$, we set
\begin{align*}
 \psi_i &= B_1^+ + \cdots + B_{i-1}^+ + \frac{1}{\alpha}B_{i}^+
               - \Big(1-\frac{1}{\alpha}\Big)( A_1^+ + \cdots +  A_{i}^+ ) ,\\
 \theta_i &= \frac{1}{\alpha}\Big( 1 + \frac{1}{g} + \cdots +\frac{1}{g^i}\Big) 
              - \frac{1}{g^i}.
\end{align*}
We will show by induction that $\psi_i \le \theta_i B^+_i$ for each of these values 
of $i$.

For $i=1$, we find, as claimed, that
\[
 \psi_1 = \frac{1}{\alpha}B_{1}^+ - \Big(1-\frac{1}{\alpha}\Big) A_1^+
           \le \left(\frac{1}{\alpha} - \frac{1}{g}\Big(1-\frac{1}{\alpha}\Big) \right) B_1^+
           = \theta_1 B^+_1
\]
using the inequality $A^+_1 \ge (1/g)B^+_1$ from \eqref{MM:prop:n-sys:eq3} and the 
fact that $\alpha>1$ by \eqref{MM:prop:n-sys:eq2}.

Suppose now that $\psi_{i-1} \le \theta_{i-1} B^+_{i-1}$ for some $i$ with 
$2\le i\le n-2$.  We find
\[
\begin{aligned}
 \psi_i &= \psi_{i-1} +  \Big(1-\frac{1}{\alpha}\Big) B^+_{i-1} 
               + \frac{1}{\alpha}B_{i}^+ - \Big(1-\frac{1}{\alpha}\Big)A^+_i \\
 &\le \Big(\theta_{i-1} + 1 - \frac{1}{\alpha}\Big)B^+_{i-1} 
               + \frac{1}{\alpha}B_{i}^+ - \Big(1-\frac{1}{\alpha}\Big)A^+_i.
\end{aligned}
\]
Since $\theta_{i-1}\ge 1/\alpha-1/g^{i-1}\ge 1/\alpha-1$, the coefficient
of $B^+_{i-1}$ in the above estimate is non-negative.  Thus, 
using $B_{i-1}^+ \le A^+_i$ from \eqref{MM:prop:n-sys:eq5}, we obtain 
\[
 \psi_i \le \theta_{i-1}A^+_i  + \frac{1}{\alpha}B_{i}^+.
\]
By \eqref{MM:prop:n-sys:eq2}, we have $\alpha > 1+ g + \cdots +g^{i-1}$,
thus $\theta_{i-1} < 0$.  Using $A^+_i\ge (1/g)B^+_i$ from 
\eqref{MM:prop:n-sys:eq3}, we deduce that
\[
 \psi_i \le \Big(\frac{1}{\alpha} + \frac{1}{g}\theta_{i-1}\Big)B_{i}^+
      = \theta_i B^+_i,
\]
which completes the induction step.

This proves that  $\psi_{n-2} \le \theta_{n-2} B^+_{n-2}$.  Since we have
$\psi_{n-2}\ge 0$ by \eqref{MM:prop:n-sys:eq4} and since $B^+_{n-2}>0$, we 
conclude that $\theta_{n-2}\ge 0$, which yields \eqref{MM:prop:n-sys:eq1}.
\end{proof}

\begin{proof}[\textbf{Proof of Theorem 2.5, part (ii)}]
As explained at the beginning of the section, we may assume that $n\ge 3$.  
Let $\cS$ denote the set of solutions $(\alpha,\beta)$ in $(0,\infty)^2$ of the 
inequalities \eqref{dr:thm:MM:eq2} in Theorem~\ref{dr:thm:MM} (ii).  By 
Proposition~\ref{dr:prop:Ri}, the spectrum $\cS_{n-1,n}$ contains $\cS$.
Conversely, in light of Corollary~\ref{pgn:tool:cor}, the preceding proposition shows 
that all points $(\alpha,\beta)$ of $\cS_{n-1,n}$ with $\beta<\infty$ are contained 
in $\cS$.   As these are dense in $\cS_{n-1,n}$ and as
$\cS_{n-1,n}$ is a closed subset of $[0,\infty]^2$, we conclude that $\cS_{n-1,n}$
is the topological closure of $\cS$ in $[0,\infty]^2$.
\end{proof}

Proposition~\ref{MM:prop:n-sys} admits the following dual statement.

\begin{proposition}
\label{MM:prop:bn-sys}
Let $\uP=(P_1,\dots,P_n)\colon (-\infty, q_0]\to\bR^n$ be a proper 
non-degenerate backwards $n$-system for some $n\ge 3$.  Suppose that 
$\chibot_1(\uP)>0$.  Then, we have
\[
  n-1 
 \ge  \chitop_1(\uP)
 \ge \sum_{i=0}^{n-2}\left(\frac{\chibot_1(\uP)}{ \chitop_1(\uP)}\right)^i.
\]
\end{proposition}

\begin{proof}
The argument is similar to the proof of Proposition~\ref{MM:prop:n-sys}.
Using Proposition \ref{part:prop:bn-sys} with $m=1$, we first note that 
$1\le  \chitop_1(\uP) \le n-1$ where the lower bound comes from the 
fact that, in the notation of that proposition, we have $A^+_i\le A^-_i\le 0$ 
for each $i\le -1$, when $m=1$.  Let $\alpha,\beta\in(0,\infty)$ with 
\[
 \frac{1}{\beta} < \chibot_1(\uP) \le \chitop_1(\uP) < \frac{1}{\alpha}.
\]  
Since $\chitop_1(\uP)\ge 1$, we have $\alpha<1$ and it remains to show that
\begin{equation}
\label{MM:prop:bn-sys:eq1}
 \frac{1}{\alpha} \ge 1+\frac{1}{g}+\cdots+\frac{1}{g^{n-2}} 
 \quad \text{where} \quad g=\frac{\beta}{\alpha}.
\end{equation}
Since $g>1$, we may assume that 
\begin{equation}
\label{MM:prop:bn-sys:eq2}
  1 > \alpha > \Big(\sum_{i=0}^\infty \frac{1}{g^i}\Big)^{-1} = 1 - \frac{1}{g}.
\end{equation}

Let $c_{-1}>c_{-2}>c_{-3}>\cdots$ denote the elements of $\cA_1(\uP)$.  
In view of formulas \eqref{part:prop:bn-sys:eq} from 
Proposition~\ref{part:prop:bn-sys}, there is an integer $r\le -2$ such that, 
for each $i\le r$, we have $P_n(c_i)<0$ and 
\[
 \frac{1}{\beta} \le \frac{P_2(c_{i+1})+\cdots+P_n(c_{i+1})}{P_1(c_i)} 
           \le \frac{P_2(c_i)+\cdots+P_n(c_i)}{P_1(c_i)} \le \frac{1}{\alpha}.
\]  
Moreover, since $\uP$ is proper, its component $P_n$ is unbounded.  Thus, there 
are infinitely many integers $i\le r$ for which $P_n(c_i)<P_n(c_{i+1})$.  For each of 
these, the simple interval $[c_i,c_{i+1}]$ in $\cI_1(\uP)$ has type $\{1,\dots,n\}$.  
Thus, there is an integer $s<r$ such that $[c_s,c_r]$ contains at least $n-2$
simple intervals of type $\{1,\dots,n\}$ from  $\cI_1(\uP)$.

We apply Lemma \ref{MM:lemma} to the subinterval $J=[c_s,c_r]$ of $(-\infty,q_0]$.
By the choice of $r$, the corresponding numbers in 
\eqref{MM:lemma:eq3} are negative and so they satisfy
\begin{equation}
\label{MM:prop:bn-sys:eq3}
 A^-_i \le \alpha A^+_i, \quad
 B^-_i \le \alpha B^+_i \et
 B^+_i\le \frac{1}{\beta} A^-_i \le \frac{1}{g} A^+_i
 \quad (1\le i  \le n-2).
\end{equation}
Using the above upper bounds for $A^-_i$ and $B^-_i$, we deduce 
from \eqref{MM:lemma:eq2} that
\begin{equation}
\label{MM:prop:bn-sys:eq4}
 (1-\alpha)( A_1^+ + \cdots +  A_{n-2}^+ ) 
 \le B_1^+ + \cdots + B_{n-3}^+ + \alpha B_{n-2}^+ \, .
\end{equation}
We also note that 
\begin{equation}
\label{MM:prop:bn-sys:eq5}
  B_{i-1}^+ \le A^+_i < 0 \quad (2\le i \le n-2)
\end{equation}
because of \eqref{MM:lemma:eq1} and the fact that $P_2+\cdots+P_n$ 
is a non-decreasing function.

For each $i=1,\dots,n-2$, we set
\begin{align*}
 \psi_i &= B_1^+ + \cdots + B_{i-1}^+ + \alpha B_{i}^+
               - (1-\alpha)( A_1^+ + \cdots +  A_{i}^+ ) ,\\
 \theta_i &= \alpha ( 1 + g + \cdots + g^i ) - g^i.
\end{align*}
We will show by induction that 
\begin{equation}
\label{MM:prop:bn-sys:eq6}
  0 \le \psi_i \le \theta_i B^+_i \quad \text{for $i=1,\dots,n-2$.}
\end{equation}

We first observe that, for each $i$ with $2 \le i \le n-2$, we have 
\begin{equation}
\label{MM:prop:bn-sys:eq7}
 \psi_i - \psi_{i-1} 
 = (1-\alpha)B_{i-1}^+ + \alpha B_i^+ - (1-\alpha)A_i^+
 \le 0
\end{equation}
using \eqref{MM:prop:bn-sys:eq2}, \eqref{MM:prop:bn-sys:eq5} and $B_i^+<0$.
Since \eqref{MM:prop:bn-sys:eq4} translates into $\psi_{n-2}\ge 0$, this gives
\[
 \psi_1 \ge \psi_2 \ge  \cdots \ge \psi_{n-2} \ge 0.
\]

We find 
$\psi_1 = \alpha B_{1}^+ - (1-\alpha) A_1^+ 
           \le (\alpha - g(1-\alpha)) B_1^+
           = \theta_1 B^+_1$
since $A^+_1 \ge g B^+_1$ and $\alpha<1$.  Thus, \eqref{MM:prop:bn-sys:eq6}
holds for $i=1$.

Suppose now that $\psi_{i-1} \le \theta_{i-1} B^+_{i-1}$ for some $i$ with 
$2\le i\le n-2$.  Since $B^+_{i-1}<0$, we deduce from this and 
\eqref{MM:prop:bn-sys:eq7} that
\begin{equation}
\label{MM:prop:bn-sys:eq8}
 \theta_{i-1}\le 0 
 \et
 \psi_i \le (\theta_{i-1} + 1 - \alpha)B^+_{i-1} + \alpha B_i^+ - (1-\alpha)A^+_i.
\end{equation}
Using \eqref{MM:prop:bn-sys:eq2}, we find
\[
\begin{aligned}
 \theta_{i-1} + 1 - \alpha
 &= (1-g^{i-1}) +\alpha g(1+g+\cdots+g^{i-2}) \\
 &= (1-g+\alpha g)(1+g+\cdots+g^{i-2}) \ge 0.
\end{aligned}
\]
So, we may use \eqref{MM:prop:bn-sys:eq5} to eliminate $B^+_{i-1}$ from the upper 
bound for $\psi_i$ in \eqref{MM:prop:bn-sys:eq8}.  This gives 
\[
 \psi_i \le \theta_{i-1}A^+_i  + \alpha B_{i}^+.
\]
Since $\theta_{i-1}\le 0$ by \eqref{MM:prop:bn-sys:eq8} and 
since $A^+_i \ge g B^+_i$ by \eqref{MM:prop:bn-sys:eq3}, we conclude that
$\psi_i \le (\alpha+g\theta_{i-1})B_{i}^+ = \theta_i B^+_i$.  This completes the
proof of \eqref{MM:prop:bn-sys:eq6} by induction on $i$.

In particular, we have $0\le \psi_{n-2} \le \theta_{n-2} B^+_{n-2}$, thus
$\theta_{n-2}\le 0$, and \eqref{MM:prop:bn-sys:eq1} follows.
\end{proof}

\begin{proof}[\textbf{Proof of Theorem 2.5, part (i)}]
Again, we may assume that $n\ge 3$.  
Let $\cS$ denote the set of solutions $(\alpha,\beta)$ in $(0,\infty)^2$ of the 
inequalities \eqref{dr:thm:MM:eq1} in Theorem~\ref{dr:thm:MM} (i).  By 
Proposition~\ref{dr:prop:Ri}, the spectrum $\cS_{1,n}$ contains $\cS$.
Conversely, in light of Corollary~\ref{pgn:duality:cor}, Proposition~\ref{MM:prop:bn-sys} 
shows that all points $(\alpha,\beta)$ of $\cS_{1,n}$ with $\beta<\infty$ are contained 
in $\cS$.   As these points are dense in $\cS_{1,n}$ and as
$\cS_{1,n}$ is a closed subset of $[0,\infty]^2$, we conclude that $\cS_{1,n}$
is the topological closure of $\cS$ in $[0,\infty]^2$.
\end{proof}

%
%

\section{Proof of Theorem \ref{dr:thm:Ri}}
\label{sec:Ri}

By Proposition~\ref{dr:prop:Ri}, the spectrum $\cS_{2,4}$ contains the set $\cS$
of all points $(\alpha,\beta)\in[0,\infty]^2$ with $1\le \alpha^2\le\beta$.   
By Corollary~\ref{pgn:tool:cor}, the next result implies the reverse inclusion and 
thus the equality of the two sets, as stated in Theorem~\ref{dr:thm:Ri}.

\begin{proposition}
\label{Ri:prop}
Let $\uP=(P_1,\dots,P_4)$ be a proper non-degenerate $4$-system.  Then
\[
 1\le \chibot_2(\uP) \et \chibot_2(\uP)^2\le \chitop_2(\uP).
\]
\end{proposition}

\begin{proof}
The first inequality follows from Proposition~\ref{part:prop:n-sys}.  To prove the 
second one, we may assume that $1<\chibot_2(\uP)\le \chitop_2(\uP)<\infty$.  Choose 
$(\alpha,\beta)\in(1,\infty)^2$ with
\[
 \alpha < \chibot_2(\uP) \le \chitop_2(\uP) < \beta.
\]
It remains to show that 
\begin{equation}
\label{Ri:prop:eq0}
  \alpha \le g \quad \text{where} \quad g=\beta/\alpha.
\end{equation}
Since $\uP$ is proper, there are infinitely many simple intervals in $\cI_2(\uP)$ 
whose type contains $1$ (those on which $P_1$ is non-constant), and infinitely 
many whose type contains $4$ (those on which $P_4$ is non-constant).  
We distinguish two cases.

\textbf{Case 1.}  Suppose that there are only finitely many simple intervals
in $\cI_2(\uP)$ of type $\{1,\dots,4\}$ (those on which both 
$P_1$ and $P_4$ are non-constant).  Then, there are infinitely many
simple intervals in $\cI_2(\uP)$ of type $\{1,2,3\}$, and infinitely many
of type $\{2,3,4\}$.  Thus, $\cA_2(\uP)$ contains arbitrarily 
large numbers
\[
 a_1 < b_1 \le a_2 <b_2
\]
with the following properties:
\begin{itemize}
\item $[a_1,b_1]$ is a simple interval in $\cI_2(\uP)$ of type $\{2,3,4\}$;
 \smallskip
\item $P_1$ and $P_4$ are constant on $[b_1,a_2]$;
 \smallskip
\item $[a_2,b_2]$ is a simple interval in $\cI_2(\uP)$ of type $\{1,2,3\}$.
\end{itemize}
Fix such a choice of numbers, and write
\begin{equation}
\label{Ri:prop:eq1}
 \uP(a_i)=(a^{(i)}_1,\dots,a^{(i)}_4) \et \uP(b_i)=(b^{(i)}_1,\dots,b^{(i)}_4)
\end{equation}
for $i=1,2$.  Then, taking into account that
\[
 a_2^{(i)}=a_3^{(i)} \et b_2^{(i)}=b_3^{(i)}
\]
for $i=1,2$, and using Proposition~\ref{type:prop:kl}, we find the system of inequalities
\[
 \begin{array}{ccccccc}
 a_1^{(1)} &=      &b_1^{(1)} &=  &a_1^{(2)} &         &b_1^{(2)} \\
 \vpp        &        &\vpp        &     &\vpp        &\dpe &\vpp        \\
 a_3^{(1)} &        &b_3^{(1)} &\le &a_3^{(2)} &        &b_3^{(2)} \\
\vpp        &\dpe &\vpp        &     &\vpp        &        &\vpp     \\
a_4^{(1)} &         &b_4^{(1)} &=  &a_4^{(2)} &=      &b_4^{(2)}
 \end{array}\ .
\]
In terms of the quantities
\begin{equation}
\label{Ri:prop:eq2}
 A^-_i=a_1^{(i)}+a_2^{(i)},  \ A^+_i=a_3^{(i)}+a_4^{(i)}, \ 
 B^-_i=b_1^{(i)}+b_2^{(i)},  \ B^+_i=b_3^{(i)}+b_4^{(i)},
\end{equation}
for $i=1,2$, this yields
\[
\begin{aligned}
 &A^+_2+(B^+_2-B^-_2) 
  =(a_3^{(2)}+a_4^{(2)})+(b_4^{(2)}-b_1^{(2)})
  \le a_4^{(2)}+b_4^{(2)} = 2b^{(1)}_4,\\
 &2B^+_1-A^+_1
 =2(b_3^{(1)}+b_4^{(1)})-(a_3^{(1)}+a_4^{(1)})
 \ge 2b_4^{(1)},
\end{aligned}
\]
thus
\begin{equation}
\label{Ri:prop:eq3}
A^+_2+(B^+_2-B^-_2) \le 2B^+_1-A^+_1.
\end{equation}
Assuming that $a_1$ is large enough, Proposition~\ref{part:prop:n-sys} yields
\begin{equation}
\label{Ri:prop:eq4}
 \alpha\le \min\Big\{\frac{A^+_i}{A^-_i}, \frac{B^+_i}{B^-_i} \Big\} 
 \et
 \frac{B^+_i}{A^-_i}\le \beta
\end{equation}
for $i=1,2$.  In particular, we have
\[
 B^-_2\le (1/\alpha)B^+_2 
 \et
 A^+_1\ge \alpha A^-_1 \ge (1/g)B^+_1.
\]
Using the above inequalities to eliminate $B^-_2$ and $A^+_1$ from \eqref{Ri:prop:eq3},
we find
\[
 A^+_2+(1-1/\alpha)B^+_2 \le (2-1/g)B^+_1.
\]
Since $\alpha>1$, $g>1$ and $0<B^+_1\le A^+_2\le B^+_2$, this yields
\[
 (2-1/\alpha)A^+_2 \le (2-1/g)A^+_2,
\]
thus $2-1/\alpha\le 2-1/g$ and so $\alpha\le g$, as claimed in \eqref{Ri:prop:eq0}.

\smallskip
\textbf{Case 2.}  Suppose on the contrary that there are infinitely many simple intervals
$[a_1,b_1]$ in $\cI_2(\uP)$ of type $\{1,\dots,4\}$.  Using the notation \eqref{Ri:prop:eq1}
for $\uP(a_1)$ and $\uP(b_1)$, we have 
\[
 a_3^{(1)}=a_2^{(1)}\le b_1^{(1)} \et a_4^{(1)}\le b_3^{(1)}=b_2^{(1)}
\]
by Proposition~\ref{type:prop:kl}.  In terms of the quantities \eqref{Ri:prop:eq2} 
with $i=1$, this implies that
\begin{equation}
\label{Ri:prop:eq5}
 A^+_1 \le B^-_1.
\end{equation}
Assuming that $a_1$ is large enough, Proposition~\ref{part:prop:n-sys} yields
\eqref{Ri:prop:eq4} with $i=1$, thus
\[
 A^+_1\ge \alpha A^-_1
 \et
 B^-_1\le (1/\alpha)B^+_1 \le g A^-_1.
\]
Substituting the above estimates into \eqref{Ri:prop:eq5}, we find
\[
\alpha A^-_1 \le g A^-_1,
\]
thus $\alpha\le g$, once again.
\end{proof}

%
%

\section{Proof of Theorem \ref{dr:thm:S35}}
\label{sec:dim5}

The proof of Theorem \ref{dr:thm:S35} is done by double inclusion.  
Each of the two propositions below proves one inclusion. 

\begin{proposition}
\label{dim5:prop:cond}
Let $\uP=(P_1,\dots,P_5)$ be a proper non-degenerate $5$-system.  
Suppose that the point $(\alpha,\beta)=(\chibot_2(\uP),\chitop_2(\uP))$ 
has $\beta<\infty$.  Then, it satisfies
\begin{equation}
\label{dim5:prop:cond:eq1}
 3/2 \le \alpha \le (\beta/\alpha)^2+1.
\end{equation}
\end{proposition}

By Corollary \ref{pgn:tool:cor}, the spectrum $\cS_{3,5}$ consists of the points 
$(\chibot_2(\uP),\chitop_2(\uP))$ where $\uP$ is a proper 
non-degenerate $5$-system.  Thus, this proposition shows that each point 
$(\alpha,\beta)$ of $\cS_{3,5}$ either has $\beta=\infty$ or else satisfies 
$\alpha\le (\beta/\alpha)^2+1$.  In particular, if $\alpha\ge 5$ and if we write
$\alpha=g^2+1$ with $g\in[2,\infty]$, then $\beta\ge g\alpha$, as in
Theorem~\ref{dr:thm:S35}.

We note also that, for $\uP$ as in the proposition, both $P_1$ and $P_5$ are 
unbounded, thus there are infinitely many simple intervals in $\cI_2(\uP)$ of type 
containing $\{1,2,3\}$, namely those on which $P_1$ is non-constant, and 
infinitely many of type containing $\{2,3,4,5\}$, namely
those on which $P_5$ is non-constant.  However, the set of simple intervals 
in $\cI_2(\uP)$ whose type contains $\{1,\dots,4\}$ may be finite or infinite.  
The proof of the proposition below considers the two cases separately.

\begin{proof}[Proof of Proposition~\ref{dim5:prop:cond}]
The first inequality in \eqref{dim5:prop:cond:eq1} follows from 
Proposition~\ref{part:prop:n-sys}.  To prove the second one, choose 
instead $\alpha,\beta\in\bR$ with
\[
 1< \alpha < \chibot_2(\uP) \le \chitop_2(\uP) <\beta.
\]
Then, by continuity, we simply need to show that $\alpha\le g^2+1$ 
where $g=\beta/\alpha$.   To this end, we may assume 
that $\alpha>g$.  We distinguish two cases.

\medskip
\textbf{Case 1.} Suppose that infinitely many simple intervals in $\cI_2(\uP)$
have type containing $\{1,\dots,4\}$.

Starting from such an interval $[a_2,b_2]$ with $a_2$ large enough, and going 
backwards, we find a simple interval $[a_1,b_1]$ with $b_1\le a_2$ maximal, 
on which $P_5$ is non-constant.   Moreover, $a_1$ goes to infinity with $a_2$.  
Thus, there are arbitrarily large elements $a_1<b_1\le a_2<b_2$ in $\cA_2(\uP)$
such that $[a_1,b_1]$ is simple of type containing $\{2,\dots,5\}$,  $P_5$ is 
constant on $[b_1,a_2]$, and $[a_2,b_2]$ is simple of type containing 
$\{1,\dots,4\}$.  Similarly as in \eqref{Ri:prop:eq1} and \eqref{Ri:prop:eq2}, we write
\begin{equation}
\label{dim5:prop:cond:eq3}
 \begin{array}{lll}
 \uP(a_i)=(a^{(i)}_1,\dots,a^{(i)}_5), &A^-_i=a_1^{(i)}+a_2^{(i)},  
        &A^+_i=a_3^{(i)}+a_4^{(i)}+a_5^{(i)}, \\[5pt]
 \uP(b_i)=(b^{(i)}_1,\dots,b^{(i)}_5), &B^-_i=b_1^{(i)}+b_2^{(i)},  
        &B^+_i=b_3^{(i)}+b_4^{(i)}+b_5^{(i)},
 \end{array}
\end{equation}
for $i=1,2$.  Among the inequalities which, by Proposition~\ref{type:prop:kl}, 
the numbers $a_j^{(i)}$ and $b_j^{(i)}$ satisfy, we retain the following ones:
\[
\begin{array}{ccccccc}
 a_1^{(1)} &&b_1^{(1)} &&a_1^{(2)} &&b_1^{(2)} \\
 &&&&&\dpe \\
 a_3^{(1)} &&b_3^{(1)} &\le&a_3^{(2)} &&b_3^{(2)} \\
 &\dpe&&&&\dpe\\
 a_4^{(1)} &&b_4^{(1)} &\le&a_4^{(2)} &&b_4^{(2)} \\
 &\dpe\\
 a_5^{(1)} &&b_5^{(1)} &=&a_5^{(2)} &&b_5^{(2)} 
\end{array}\ .
\]
For the numbers $A^-_i, A^+_i,B^-_i, B^+_i$, we deduce that
\begin{align*}
 A^+_1-A^-_1 &= a_4^{(1)}+a_5^{(1)}-a_1^{(1)} 
     \le a_4^{(1)}+a_5^{(1)} \le b_3^{(1)}+b_4^{(1)} = B^+_1-b_5^{(1)},\\
 A^+_2 &= a_3^{(2)}+a_4^{(2)}+a_5^{(2)} 
     \le b_1^{(2)}+b_3^{(2)}+b_5^{(1)} = B^-_2+b_5^{(1)}.
\end{align*}
Summing these estimates term by term, we obtain
\begin{equation}
\label{dim5:prop:cond:eq4}
 (A^+_1-A^-_1)+A^+_2 \le B^+_1 + B^-_2.
\end{equation}
Moreover, we may assume that $a_1$ is large enough so that, by 
Proposition~\ref{part:prop:n-sys} and our choice of $\alpha$ and $\beta$, the 
inequalities \eqref{Ri:prop:eq4} hold for $i=1,2$.  This means that
\begin{equation}
\label{dim5:prop:cond:eq5}
  A^+_i\ge \alpha A^-_i, \quad 
  B^+_i\ge \alpha B^-_i \et
  B^+_i\le \beta A^-_i \le g A^+_i,
\end{equation}
for $i=1,2$.   Using $B^+_2 \ge \alpha B^-_2$ to eliminate $B^-_2$ from
\eqref{dim5:prop:cond:eq4}, we find
\begin{equation}
\label{dim5:prop:cond:eq6}
 (A^+_1-A^-_1)+A^+_2 \le B^+_1 + (1/\alpha) B^+_2.
\end{equation}
Multiplying both sides of the above by $\alpha$, the estimates 
$\alpha A^-_1 \le A^+_1$ and $B^+_2 \le g A^+_2$ yield
\[
 (\alpha-1)A^+_1 + (\alpha-g)A^+_2 \le \alpha B^+_1.
\]
Since $\alpha-g >0$, we may use $A^+_2 \ge B^+_1$ to eliminate $A^+_2$.  This gives
\[
 (\alpha-1)A^+_1 \le g B^+_1.
\]
Finally, since $B^+_1\le g A^+_1$, we conclude that $\alpha-1 \le g^2$, as needed.

\medskip
\textbf{Case 2.} Suppose on the contrary that only finitely many simple 
intervals in $\cI_2(\uP)$ have type containing $\{1,\dots,4\}$.

Then any simple interval $[a_3,b_3]$ in $\cI_2(\uP)$ on which $P_1$ is 
non-constant has type $\{1,2,3\}$, if $a_3$ is large enough.  Starting from 
such an interval  and going backwards, we find a simple interval $[a_2,b_2]$ with 
$b_2\le a_3$ maximal, on which $P_4$ is non-constant, and then a simple interval
$[a_1,b_1]$ with $b_1\le a_2$ maximal, on which $P_5$ is non-constant. 
Moreover, $a_1$ goes to infinity with $a_3$.  Thus, there are arbitrarily large 
elements $a_1<b_1\le a_2<b_2\le a_3<b_3$ in $\cA_2(\uP)$
such that $[a_1,b_1]$ is simple of type containing $\{2,\dots,5\}$,  $P_5$ is 
constant on $[b_1,a_2]$, $[a_2,b_2]$ is simple of type containing 
$\{2,3,4\}$, $P_4$ and $P_5$ are constant on $[b_2,a_3]$, and $[a_3,b_3]$ 
is simple of type $\{1,2,3\}$.  In the notation \eqref{dim5:prop:cond:eq3}, 
this implies the following inequalities
\[
\begin{array}{ccccccccccc}
 a_1^{(1)} &&b_1^{(1)} &&a_1^{(2)} &&b_1^{(2)} &&a_1^{(3)} &&b_1^{(3)} \\
 &&&&&&&&&\dpe \\
 a_3^{(1)} &&b_3^{(1)} &&a_3^{(2)} &&b_3^{(2)} &\le&a_3^{(3)} &&b_3^{(3)} \\
 &\dpe&&&\vpe&\dpe\\
 a_4^{(1)} &&b_4^{(1)} &&a_4^{(2)} &&b_4^{(2)} &=&a_4^{(3)} &=&b_4^{(3)}\\
 &\dpe\\
 a_5^{(1)} &&b_5^{(1)} &=&a_5^{(2)} &&b_5^{(2)} &=&a_5^{(3)} &=&b_5^{(3)}
\end{array}
\ ,
\]
and then
\begin{align*}
 A^+_1-A^-_1 &\le B^+_1-b_5^{(1)} = B^+_1-a_5^{(2)} ,\\
 A^+_2 &\le a_5^{(2)}+2b_3^{(2)} ,\\
 B^+_3-B^-_3 &= b_4^{(3)}+b_5^{(3)}-b_1^{(3)}  
    \le b_4^{(2)}+b_5^{(2)}-b_3^{(2)} = B^+_2-2b_3^{(2)} .
\end{align*}
Summing these three inequalities term by term, we obtain
\begin{equation}
\label{dim5:prop:cond:eq7}
 (A^+_1-A^-_1)+A^+_2 +(B^+_3-B^-_3) \le B^+_1 + B^+_2.
\end{equation}
As in case 1, we may assume that $a_1$ is large enough so that 
the estimates in \eqref{dim5:prop:cond:eq5} hold for $i=1,2,3$.
In particular, we have $B^-_3\le (1/\alpha) B^+_3$ and 
$B^+_3\ge B^+_2$, thus
\[
B^+_3-B^-_3 \ge (1-1/\alpha)B^+_3 \ge (1-1/\alpha)B^+_2.
\]
Substituting this into \eqref{dim5:prop:cond:eq7}, we obtain the same 
estimate \eqref{dim5:prop:cond:eq6} as in the previous situation,
and we conclude once again that $\alpha\le g^2+1$. 
\end{proof}

\begin{proposition}
\label{dim5:prop:cons}
Let $s\ge 3$ be an integer and let $g\in[2,\infty)$.  Then $\cS_{3,5}$ contains the points 
$(\alpha,\beta)\in \bR^2$ with
\[
 \alpha=g^2+1-\frac{g}{1+g^s} \et \beta\ge g\alpha.
\]
\end{proposition}

This completes the proof of Theorem \ref{dr:thm:S35} because, as $\cS_{3,5}$
is a closed subset of $[0,\infty]^2$, we deduce, by letting $s$ go to infinity,
that $\cS_{3,5}$ contains the points $(\alpha,\beta)\in [5,\infty]^2$ with
$\alpha=g^2+1$ and $\beta\ge g\alpha$ for each $g\in[2,\infty)$, and so 
for each $g\in[2,\infty]$.

\begin{proof}[Proof of Proposition \ref{dim5:prop:cons}]
For the given $s$ and $g$, let
\[
 \rho =g^s, \quad 
 \alpha =g^2+1-g/(1+\rho), \quad
 b_i = g^{i-1}(1+\rho)-\rho \quad\text{for $i=2,\dots,s+1$,}
\]
and define recursively
\[
 d_1 =\alpha(1+\rho)-\rho-b_2 \et 
 d_i =\alpha(\rho+b_i)-b_i-d_{i-1} \quad\text{for $i=2,\dots,s-1$.}
\]
Note that $b_{s+1}=\rho^2$.  We claim that the points 
\begin{align*}
 \ua^{(1)} &= (1,\rho,\rho, b_2,d_1),\\
 \ua^{(i)} &= (\rho,b_i, b_i,d_{i-1},d_i) \quad \text{for $i=2,\dots,s-1$,}\\
 \ua^{(s)} &= (\rho,b_s, b_s,\rho^2,\rho b_2),\\
 \ua^{(s+1)} &= (\rho,\rho^2,\rho^2, \rho b_2,\rho d_1) = \rho\ua^{(1)},
\end{align*}
satisfy the hypotheses of Theorem~\ref{cons:thm} with $m=2$, $n=5$, 
\[
 (k_1,\ell_1)=(1,5) \et (k_2,\ell_2)=\cdots=(k_s,\ell_s)=(2,5),
\]
as illustrated below.  
%
\[
\begin{array}{ccccccccccc}
 1 &&\rho &=&\rho &=\ \cdots\ =&\rho &=&\rho&=&\rho \\
\vpp &\dpe&\vpp&&\vpp&&\vpp&&\vpp&&\vpp \\
 \rho &&b_2 &&b_3 &&b_{s-1} &&b_s &&\rho^2 \\
 \veq&\dpe&\veq&\dpe&\veq&\dpe\ \cdots\ \dpe&\veq&\dpe&\veq&\dpe&\veq\\
 \rho &&b_2 &&b_3 &&b_{s-1} &&b_s &&\rho^2 \\
 \vpp&\dpe&\vpp&\dpe&\vpp&\dpe\ \cdots\ \dpe&\vpp&\dpe&\vpp&\dpe&\vpp\\
 b_2 &&d_1 &&d_2 &&d_{s-2} &&\rho^2 &&\rho b_2 \\
 \vpp&\dpe&\vpp&\dpe&\vpp&\dpe\ \cdots\ \dpe&\vpp&\dpe&\vpp&\dpe&\vpp\\
 d_1 &&d_2 &&d_3 &&d_{s-1} &&\rho b_2 &&\rho d_1 \\
\end{array}
\]
If we admit this claim, the conclusion follows easily.  Indeed, in the notation
\eqref{cons:thm:eq2} (with $m=2$ and $n=5$), we find
\[
 A^-_i=g^{i-1}(1+\rho) \et A^+_i=\alpha A^-_i \quad \text{for $i=1,\dots,s+1$.}
\]
This is immediate from the definitions of the numbers $\rho$, $b_i$ and $d_i$, except for 
$A^+_s$ which also depends on the choice of $\alpha$:
\begin{align*}
 A^+_s = b_s+\rho^2+\rho b_2 
     &= g^{s-1}(1+\rho) - \rho + \rho g (1+\rho) \\
     &= A^-_s (1-g/(1+\rho)+g^2)=\alpha A^-_s.
\end{align*}
Thus, for each $i=1,\dots,s$, we have $A^+_i/A^-_i=\alpha$ and $A^+_{i+1}/A^-_i =
\alpha A^-_{i+1}/A^-_i = g\alpha$, and so, by Theorem~\ref{cons:thm}, the point
$(\alpha,g\alpha)$ belongs to $\cS_{3,5}$.  By Theorem~\ref{dr:thm:rectangles}, 
it follows that $\cS_{3,5}$ contains all points $(\alpha,\beta)$
with $\beta\in [g\alpha,\infty]$.

To prove our claim, we first note that, since $g\ge 2$, we have
\[ 
 b_2 =g(1+\rho)-\rho> \rho >1.
\]
To complete the proof, it suffices to show the stronger inequalities
\begin{equation}
\label{dim5:prop:cons:eq}
 \rho + b_i < d_{i-1} < b_{i+1} \quad \text{for $i=2,\dots,s$,}
\end{equation}
as $b_{s+1}=\rho^2$.  We prove this by induction on $i$, using the fact that
\[
 2g \le g^2 < \alpha < g^2+1 \le g^2+g-1
\]
Since $d_1=(\alpha-g)(1+\rho)$, we find 
\[
 \rho+b_2 = g(1+\rho) < d_1 < (g^2-1)(1+\rho) = b_3-1 < b_3.
\]
Thus, \eqref{dim5:prop:cons:eq} holds for $i=2$.  Suppose, by induction, that it holds for
some integer $i$ with $2\le i<s$.  Since $0 < g^i < \rho < b_2 \le b_i$ and
\[
 d_i + d_{i-1} - \rho = (\alpha-1)(\rho+b_i) 
   = (g^2-g/(1+\rho))g^{i-1}(1+\rho) = g^{i+1}(1+\rho) - g^i,
\]
we find
\begin{align*}
 d_i &> g^{i+1}(1+\rho) - d_{i-1} > 2(\rho+b_{i+1}) -b_{i+1} > \rho+b_{i+1},\\
 d_i&< g^{i+1}(1+\rho) +\rho - d_{i-1} < 2\rho+b_{i+2}- (\rho+b_i) < b_{i+2},
\end{align*}
which completes the induction step.
\end{proof}

%
%

\section{More on the spectrum $\cS_{3,5}$}
\label{sec:more}

We have seen in the preceding section, as a consequence of 
Proposition~\ref{dim5:prop:cond}, that any point $(\alpha,\beta)$ in
$\cS_{3,5}$ either has $\beta=\infty$ or satisfies $\alpha\le g^2+1$
where $g=\beta/\alpha$.  In the same way, the next result below shows 
that any point $(\alpha,\beta)$ in $\cS_{3,5}$ with 
$\alpha\ge 2(1+1/\sqrt{3})$ either has $\beta=\infty$ or satisfies 
$\alpha\le (3/2)g^2-g+1$ where $g=\beta/\alpha$.  The latter lower bound 
on $g$ (thus on $\beta$)  is stronger than the former if and only if $\alpha\le 5$.  
We will show at the end of the section that this inequality describes 
the portion of $\cS_{3,5}$ in $I\times[0,\infty]$ for a small subinterval 
$I$ of $[2(1+1/\sqrt{3}),5]$.  It would be interesting to know the largest 
interval $I$ with this property.

\begin{proposition}
\label{more:prop1}
Let $\uP=(P_1,\dots,P_5)$ be a proper non-degenerate $5$-system.  
Suppose that $(\alpha,\beta)=(\chibot_2(\uP),\chitop_2(\uP))$ 
satisfies $\beta<\infty$ and $2(1+1/\sqrt{3})\le \alpha$.  Then, 
\[
 \alpha \le (3/2)(\beta/\alpha)^2-\beta/\alpha+1.
\]
\end{proposition}

The proof of this result is similar to that of Proposition \ref{dim5:prop:cond} 
but requires a more complex analysis involving several lemmas.  For $\uP$ as above, 
we fix a choice of numbers $\alpha$, $\beta$ with
\begin{equation}
\label{more:prop1:eqab}
 1 < \alpha < \chibot_2(\uP) \le \chitop_2(\uP) <\beta,
\end{equation}
and set $g=\beta/\alpha$.  We will consider finite subsequences of $\cA_2(\uP)$ 
of the form
\begin{equation}
\label{more:prop1:seq}
 a_1<b_1\le a_2<b_2\le \cdots \le a_s<b_s
\end{equation}
where each $[a_i,b_i]$ is a simple interval in $\cI_2(\uP)$ and we will denote
by $a_j^{(i)}$, $b_j^{(i)}$, $A^\pm_i$, $B^\pm_i$ the corresponding 
numbers given by \eqref{dim5:prop:cond:eq3} for $i=1,\dots,s$ 
and $j=1,\dots,5$.  In view of Proposition~\ref{part:prop:n-sys} and
the choice of $\alpha$, $\beta$, there is a number $a_0$ in $\cA_2(\uP)$ 
such that, if $a_1\ge a_0$, then the inequalities \eqref{dim5:prop:cond:eq5} 
hold for $i=1,\dots,s$.  We fix such a choice of $a_0$ and define  
\[
 \cA^* = \{ a\in\cA_2(\uP) \,;\, a\ge a_0 \}.
\]
To ensure that \eqref{dim5:prop:cond:eq5} applies, we will take 
the subsequences \eqref{more:prop1:seq} inside $\cA^*$. 

\begin{definition}
\label{more:def}
A \emph{bloc} in $\cA^*$ is a sequence $a_1<b_1\le a_2<b_2$ in $\cA^*$ such 
that $[a_1,b_1]$ and $[a_2,b_2]$ are simple intervals in $\cI_2(\uP)$ of type 
containing $\{2,\dots,5\}$, and such that $P_5$ is constant on $[b_1,a_2]$.
We say that a sequence of the form \eqref{more:prop1:seq} with $s\ge 2$ 
consists of consecutive blocs if $a_i<b_i\le a_{i+1}<b_{i+1}$ is a bloc for 
each $i=1,\dots,s-1$.
\end{definition}

\begin{lemma}
\label{more:lemma1}
Let $a_1<b_1\le a_2<b_2\le a_3<b_3$ be consecutive blocs in $\cA^*$.
Suppose that $[a_i,b_i]$ has type $\{1,\dots,5\}$ for $i=1,2,3$. 
Then we have $\alpha\le (g^2+g+1)/(1+1/g)$. 
\end{lemma}

\begin{proof}
We may assume that $\alpha > g$.  By hypothesis, $[a_i,b_i]$ is a simple interval of
type $\{1,\dots,5\}$ for $i=1,2,3$, and $P_5$ is constant on $[b_i,a_{i+1}]$
for $i=1,2$.  Thus, we have
\[
\begin{array}{ccccccccccc}
 a_1^{(1)} &&b_1^{(1)} &\le&a_1^{(2)} &&b_1^{(2)} &&a_1^{(3)} &&b_1^{(3)} \\
 &\dpe&&&&\dpe&&&&\dpe \\
 a_3^{(1)} &&b_3^{(1)} &&a_3^{(2)} &&b_3^{(2)} &&a_3^{(3)} &&b_3^{(3)} \\
 &\dpe&&&&\dpe&&&&\dpe\\
 a_4^{(1)} &&b_4^{(1)} &&a_4^{(2)} &&b_4^{(2)} &&a_4^{(3)} &&b_4^{(3)}\\
 &\dpe&&&&\dpe\\
 a_5^{(1)} &&b_5^{(1)} &=&a_5^{(2)} &&b_5^{(2)} &=&a_5^{(3)} &&b_5^{(3)}
\end{array}
\ ,
\]
which yields
\[
 A^+_1+(A^+_2-A^-_2) \le B^+_1+b_3^{(2)}, \quad
 A^+_2  \le B^-_2+b_4^{(2)}, \quad
 A^+_3 \le B^-_3+b_5^{(2)},
\]
and thus,
\[
 A^+_1+2A^+_2-A^-_2+A^+_3 \le B^+_1+B^-_2+B^+_2+B^-_3.
\]
We eliminate $A^-_2$, $B^-_2$, $A^+_1$ and $B^-_3$ from this inequality using 
$A^-_2\le (1/\alpha)A^+_2$, $B^-_2\le(1/\alpha)B^+_2$,
$A^+_1\ge(1/g)B^+_1$ and $B^-_3\le(1/\alpha)B^+_3\le(g/\alpha)A^+_3$.
This gives
\[
 (2-1/\alpha)A^+_2 + (1-g/\alpha)A^+_3 \le (1-1/g)B^+_1 + (1+1/\alpha)B^+_2.
\]
Since $\alpha>g>1$, we have $1-g/\alpha>0$ and $1-1/g>0$. So, we may use 
$A^+_3\ge B^+_2$ and  $B^+_1\le A^+_2$ to eliminate $A^+_3$ and $B^+_1$.  
After simplifications, we find
\[
 (1+1/g-1/\alpha)A^+_2 \le (1/\alpha)(g+1)B^+_2.
\]
As $B^+_2\le g A^+_2$, this implies that $1+1/g-1/\alpha \le (1/\alpha)(g^2+g)$,
and the conclusion follows.
\end{proof}

\begin{lemma}
\label{more:lemma2}
Let $a_1<b_1\le a_2<b_2$ be a bloc in $\cA^*$.
Suppose that the type of $[b_1,a_2]$ is $\{1,\dots,4\}$.   Then, we have 
$\alpha\le 2g$.  
\end{lemma}

\begin{proof}
Since $[a_1,b_1]$ is a simple interval of type containing $\{2,\dots,5\}$, 
we have
\[
\begin{array}{ccccccc}
 a_1^{(1)} &&b_1^{(1)} &&a_1^{(2)}  \\
 &&&\dpe\\
 a_3^{(1)} &&b_3^{(1)} &&a_3^{(2)}  \\
 \vpe&\dpe&&\dpe\\
 a_4^{(1)} &&b_4^{(1)} &&a_4^{(2)} \\
 &\dpe&&&\vpe\\
 a_5^{(1)} &&b_5^{(1)} &=&a_5^{(2)} 
\end{array}
\ .
\]
We find
\begin{equation}
\label{more:lemma2:eq1}
 A^+_1-A^-_1 \le a_4^{(1)}+a_5^{(1)} \le A^-_2,
\end{equation}
and also,
\[
 A^+_1 \le 2b_3^{(1)}+b_4^{(1)}, \quad
 A^+_1-A^-_1 \le b_3^{(1)}+b_4^{(1)}, \quad
 A^+_2-A^-_2 \le 2b_5^{(1)}-b_3^{(1)},
\]
thus,
\begin{equation}
\label{more:lemma2:eq2}
 A^+_1 + (A^+_1-A^-_1) + (A^+_2-A^-_2) \le 2B^+_1.
\end{equation}
Using $A^+_1\ge \alpha A^-_1$,  $A^+_2\ge \alpha A^-_2$ 
and $B^+_1\le gA^+_1$ to 
eliminate $A^-_1$, $A^-_2$ and $B^+_1$ from \eqref{more:lemma2:eq1} 
and \eqref{more:lemma2:eq2}, we get
\begin{align*}
 (\alpha-1)A^+_1  &\le A^+_2, \\
 (\alpha-1)A^+_2  &\le  (2g\alpha-2\alpha+1)A^+_1.
\end{align*}
We conclude that $(\alpha-1)^2\le 2g\alpha-2\alpha+1$ and so $\alpha\le 2g$.
\end{proof}

\begin{lemma}
\label{more:lemma3}
Let $a_1<b_1\le a_2<b_2\le a_3<b_3$ be consecutive blocs
in $\cA^*$.  Suppose that $P_1$ is non-constant on $[a_1,a_2]$, and that 
$[b_2,a_3]$ has type $\{1,2,3\}$.  Then we have $\alpha\le g(g/2+1)$.
\end{lemma}

\begin{proof}
By Proposition~\ref{type:prop:kl}, the hypotheses yield
\[
a_3^{(1)}\le a^{(2)}_1 \et 
\begin{array}{ccccccccc}
 a_1^{(1)} &&b_1^{(1)} &&a_1^{(2)} &&b_1^{(2)} &&a_1^{(3)}\\
 &&&&&&&\dpe\\
 a_3^{(1)} &&b_3^{(1)} &&a_3^{(2)} &&b_3^{(2)} &&a_3^{(3)} \\
 &\dpe&&&&\dpe\\
 a_4^{(1)} &&b_4^{(1)} &&a_4^{(2)} &&b_4^{(2)} &=&a_4^{(3)} \\
 &\dpe\\
 a_5^{(1)} &&b_5^{(1)} &=&a_5^{(2)} &&b_5^{(2)} &=&a_5^{(3)} 
\end{array}
\ ,
\]
thus
\[
 A^+_1\le B^+_1-b^{(1)}_5+a^{(2)}_1,\quad
 A^+_2-A^-_2\le b^{(1)}_5+b^{(2)}_3-a^{(2)}_1, \quad
 A^+_3-A^-_3\le B^+_2-2b^{(2)}_3,
\]
and so
\[
 2A^+_1 + 2(A^+_2-A^-_2) + (A^+_3-A^-_3) \le 2B^+_1 + B^+_2.
\]
We multiply both sides of this inequality by $\alpha$ and use
$\alpha A^-_2\le A^+_2$,  $\alpha A^-_3\le A^+_3$,  $B^+_1\le g A^+_1$
to eliminate $A^-_2$, $A^-_3$ and $A^+_1$.  This gives
\[
 2(\alpha-1)A^+_2 + (\alpha-1)A^+_3 
       \le 2\alpha(1-1/g)B^+_1 + \alpha B^+_2.
\]
Then, using $B^+_1\le A^+_2$ and $B^+_2\le A^+_3$ 
to eliminate $B^+_1$ and $A^+_3$, we obtain
\[
 2(\alpha/g-1)A^+_2  \le B^+_2.
\]
As $B^+_2\le g A^+_2$, we conclude that $2(\alpha/g-1)\le g$ and so
$\alpha\le g(g/2+1)$.
\end{proof}

\begin{lemma}
\label{more:lemma4}
Let $a_1<b_1\le a_2<b_2$ be a bloc in $\cA^*$.  Suppose that $P_1$ 
is constant on $[a_1,a_2]$, and that $[a_2,b_2]$ has type $\{1,\dots,5\}$.  
Then we have $\alpha\le (3/2)g^2-g+1$.
\end{lemma}

\begin{proof}
By the hypotheses, the type of $[a_1,b_1]$ is $\{2,\dots,5\}$ and we have
\[
\begin{array}{ccccccc}
 a_1^{(1)} &=&b_1^{(1)} &=&a_1^{(2)} &&b_1^{(2)}\\
 &&&&&\dpe&\vpe\\
 a_3^{(1)} &&b_3^{(1)} &&a_3^{(2)} &&b_3^{(2)} \\
 &\dpe&&&&\dpe\\
 a_4^{(1)} &&b_4^{(1)} &&a_4^{(2)} &&b_4^{(2)}\\
 &\dpe\\
 a_5^{(1)} &&b_5^{(1)} &=&a_5^{(2)} &&b_5^{(2)} 
\end{array}
\ ,
\]
thus
\[
 A^+_1-A^-_1 \le B^+_1-a^{(2)}_5-a^{(2)}_1,\quad
 A^-_2\le a^{(2)}_1 + (1/2)B^-_2, \quad
 A^+_2\le B^-_2+ a^{(2)}_5,
\]
and so
\[
 (A^+_1-A^-_1) + A^-_2 + A^+_2 \le B^+_1 + (3/2)B^-_2.
\]
Using $B^-_2 \le (1/\alpha)B^+_2 \le g A^-_2$, this yields
\begin{equation}
\label{more:lemma4:eq}
 2(A^+_1-A^-_1) + 2A^+_2 \le 2B^+_1 + (3g-2)A^-_2.
\end{equation}
Using $A^-_i\le (1/\alpha)A^+_i$ for $i=1,2$ to eliminate 
$A^-_1$ and $A^-_2$ from this inequality, we obtain
\[
 2(\alpha-1)A^+_1 + (2\alpha+2-3g) A^+_2 \le 2\alpha B^+_1.
\]
We now eliminate $A^+_1$ using $B^+_1\le g A^+_1$, to get
\[
 g(2\alpha+2-3g) A^+_2 \le 2((g-1)\alpha+1) B^+_1.
\] 
Since $g>1$ and $B^+_1\le A^+_2$, we conclude that 
$g(2\alpha+2-3g) \le 2((g-1)\alpha+1)$, which simplifies to
$\alpha\le (3/2)g^2-g+1$. 
\end{proof}

\begin{lemma}
\label{more:lemma5}
Let $a_1<b_1\le a_2<b_2\le a_3<b_3$ be consecutive blocs
in $\cA^*$.  Suppose that $P_1$ is constant on $[a_1,b_2]$, and that 
$[b_2,a_3]$ has type $\{1,2,3\}$.  Then we have $\alpha\le (3/2)g^2-g+1$.
\end{lemma}

\begin{proof}
By hypothesis, both $[a_1,b_1]$ and $[a_2,b_2]$ have type $\{2,\dots,5\}$.
We note that
\[
\begin{array}{ccccccccc}
 a_1^{(1)} &=&b_1^{(1)} &=&a_1^{(2)} &=&b_1^{(2)} &&a_1^{(3)} \\
 &&&&&&&\dpe&\\
 a_3^{(1)} &&b_3^{(1)} &&a_3^{(2)} &&b_3^{(2)} &&a_3^{(3)}\\
 &\dpe&&&\vpe&\dpe&&&\\
 a_4^{(1)} &&b_4^{(1)} &&a_4^{(2)} &&b_4^{(2)} &=&a_4^{(3)}\\
 &\dpe\\
 a_5^{(1)} &&b_5^{(1)} &=&a_5^{(2)} &&b_5^{(2)} &=&a_5^{(3)}
\end{array}
\ ,
\]
thus,
\[
 A^+_1-A^-_1 \le B^+_1 - B^-_2 - a^{(2)}_5 + b^{(2)}_3,\quad
 A^+_2 \le 2b^{(2)}_3 + a^{(2)}_5, \quad
 A^+_3-A^-_3 \le B^+_2 -  2b^{(2)}_3,
\]
and so
\[
 2(A^+_1-A^-_1) + 2A^+_2 +3(A^+_3-A^-_3) \le 2(B^+_1 - B^-_2) + 3B^+_2.
\]
Since $\alpha A^-_3\le A^+_3$ and $B^+_2\le A^+_3$, we also have
\[
 A^+_3-A^-_3 \ge (1-1/\alpha)A^+_3 \ge (1-1/\alpha)B^+_2.
\]
Combining the last two estimates, we obtain
\[
 2(A^+_1-A^-_1) + 2A^+_2 \le 2(B^+_1 - B^-_2) + (3/\alpha)B^+_2.
\]
Since $B^-_2\ge A^-_2$ and $B^+_2\le \beta A^-_2$, this yields \eqref{more:lemma4:eq}
and the conclusion follows by the same computations.
\end{proof}

\begin{lemma}
\label{more:lemma6}
Suppose that $\alpha >\max\{2g, \, (3/2)g^2-g+1\}$. Then there is no
bloc $a_1<b_1\le a_2<b_2$ in $\cA^*$ such that $P_1$ is constant on
$[a_1,a_2]$.  
\end{lemma}

\begin{proof}
Otherwise, since $P_1$ is unbounded, there are consecutive blocs 
$a_1<b_1\le a_2<b_2\le a_3<b_3$ in $\cA^*$ such that $P_1$ is constant 
on $[a_1,a_2]$, but non-constant on $[a_2,a_3]$.   Then, by 
Lemma~\ref{more:lemma4}, the type of $[a_2,b_2]$ is $\{2,\dots,5\}$.
Thus, $P_1$ is constant on $[a_1,b_2]$ but non-constant on $[b_2,a_3]$.  
In particular, the type of $[b_2,a_3]$ is $\{1,2,3\}$ or $\{1,2,3,4\}$.  This is
impossible as the first case is ruled out by Lemma~\ref{more:lemma5},
while the second is excluded by Lemma~\ref{more:lemma2}.
\end{proof}

\begin{lemma}
\label{more:lemma7}
Suppose that $\alpha >\max\{2g, \, (3/2)g^2-g+1\}$. Then, for each
bloc $a_1<b_1\le a_2<b_2$ in $\cA^*$, the type of $[a_2,b_2]$ is
$\{1,\dots,5\}$.
\end{lemma}

\begin{proof}
Let $a_1<b_1\le a_2<b_2\le a_3<b_3$ be consecutive blocs in $\cA^*$.
By Lemma~\ref{more:lemma6},  $P_1$ is non-constant 
on $[a_1,a_2]$ and non-constant on $[a_2,a_3]$.  Since
\[
 \alpha > (3/2)g^2-g+1 = g(g/2+1) + (g-1)^2 \ge g(g/2+1),
\]
Lemma~\ref{more:lemma3} implies that $[b_2,a_3]$ does not have 
type $\{1,2,3\}$.  By Lemma~\ref{more:lemma2}, it does not have 
type $\{1,2,3,4\}$ either.   So, $P_1$ is constant on $[b_2,a_3]$,
and thus it is non-constant on $[a_2,b_2]$.  Consequently,
the type of $[a_2,b_2]$ is $\{1,\dots,5\}$.
\end{proof}

\begin{lemma}
\label{more:lemma8}
We have $\alpha \le \max\{2g, \, (3/2)g^2-g+1\}$. 
\end{lemma}

\begin{proof}
Otherwise, Lemma~\ref{more:lemma7} shows that the hypotheses of 
Lemma~\ref{more:lemma1} are fulfilled for all but at most one pair 
of consecutive blocs in $\cA^*$, and so we have 
\[
 \alpha \le \frac{g^2+g+1}{1+1/g}  = g^2+\frac{g}{g+1}.
\]
If $g\ge 2$, this yields $\alpha< g^2+1\le (3/2)g^2-g+1$.  If $g<2$,
it gives  
\[
   \alpha < g^2+2/3 = (2/3)\big((3/2)g^2-g+1\big)+(1/3)(2g).
\]
In both cases, this is a contradiction.
\end{proof}

\begin{proof}[\textbf{Proof of Proposition~\ref{more:prop1}}]
Letting $\alpha$ go to $\chibot_2(\uP)$ and $\beta$ go to 
$\chitop_2(\uP)$ in \eqref{more:prop1:eqab}, we deduce from 
Lemma~\ref{more:lemma7} that
\[
 \chibot_2(\uP) \le \max\{2g_0,\, (3/2)g_0^2-g_0+1\}
 \quad\text{where $g_0=\chitop_2(\uP)/\chibot_2(\uP)$.}
\]
If $2g_0 \le (3/2)g_0^2-g_0+1$, we are done.  Otherwise, we have
$g_0 < 1+1/\sqrt{3}$ and then $\chibot_2(\uP) \le 2g_0 < 2(1+1/\sqrt{3})$, 
against the hypothesis.
\end{proof}

The next result shows that the boundary of $\cS_{3,5}$ contains the 
arc of curve given by \eqref{dr:eq:arc}.

\begin{proposition}
\label{more:prop2}
Let $g\in[1.798,1.839]$ and let $\alpha=(3/2)g^2-g+1$.  Then,
\[
 \cS_{3,5} \cap\big( \{\alpha\}\times [0,\infty]\big) = \{\alpha\}\times[g\alpha,\infty].
\]
\end{proposition}

\begin{proof}
Proposition \ref{more:prop1} implies that, if a pair $(\alpha,\beta)$ 
belongs to $\cS_{3,5}$, then $\beta\ge g\alpha$.   To prove the converse, 
set
\begin{align*}
 \rho&=g/(2-g),  & a&=2(\alpha-g)/g, &  b&=2g-1, \\
 c&=2(\alpha-g), & c'&=(\rho+1-g)/g, & d&=2\alpha(g-1)+1.
\end{align*}
It is useful to note that these six functions of $g$ formally 
evaluate to $1$ when $g=1$ (although this is outside of the 
range for $g$).
We claim that for $\epsilon>0$ small enough, the points 
\begin{align*}
 \ua^{(1)} &= (1/\rho,1,1,1+\epsilon, a), &
 \ua^{(2)} &= (1,1+\epsilon,1+\epsilon, b, c), &\\
 \ua^{(3)} &= (1,b,b,c,d), &
 \ua^{(4)} &= (1,c',c',\rho, \rho(1+\epsilon)), \\
 \ua^{(5)} &= (1, \rho, \rho, \rho(1+\epsilon), \rho a) = \rho\ua^{(1)}
\end{align*}
satisfy the hypotheses of Theorem~\ref{cons:thm} with $m=2$, $n=5$, 
$(k_1,\ell_1)=(1,5)$ and $(k_i,\ell_i)=(2,5)$ for $i=2,3,4$.  This amounts to showing
\[
 \rho>1, \quad 1<a\le b<c<d\le \rho \et c\le c'<\rho.
\]
Except for $d\le \rho$ and $c\le c'$, all these inequalities hold relying only upon
$1<g<2$.   As $1<g<2$, the condition $c\le c'$ reduces to $3g^3-7g^2+4g-2\ge 0$.  
It holds since $g\ge 1.798$.  Then, we find $\rho-d=g(c'-c)+g(g-1) \ge g(g-1)>0$, 
and so $d\le \rho$.

In the notation of \eqref{cons:thm:eq2} (with $m=2$ and $n=5$), we find
\begin{align*}
 A^+_1 &= 2\alpha/g+\epsilon, &
 A^+_2 &= 2\alpha+\epsilon, &
 A^+_3 &= 2g\alpha, &
 A^+_4 &= \alpha(\rho+1)/g+\rho\epsilon, & 
 A^+_5 &= 2\alpha\rho/g+\rho\epsilon, &\\
 A^-_1 &= 2/g, &
 A^-_2 &= 2+\epsilon, &
 A^-_3 &= 2g, &
 A^-_4 &= (\rho+1)/g. 
\end{align*}
Set $\talpha=\min\{A^+_1/A^-_1,\dots,A^+_4/A^-_4\}$ and 
$\tbeta=\max\{A^+_2/A^-_1,\dots,A^+_5/A^-_4\}$.  By 
Theorem~\ref{cons:thm}, the set $\cS_{3,5}$ contains 
$(\talpha,\tbeta)$.   However, as $\epsilon$ tends to 
zero, $A^+_i/A^-_i$ tends to $\alpha$ for $i=1,\dots,4$,   
$A^+_{i+1}/A^-_i$ tends to $g\alpha$ for $i=1,2,4$, and $A^+_4/A^-_3$ 
tends to $\alpha(\rho+1)/(2g^2)$.   As $1<g<2$, the condition 
$(\rho+1)/(2g^2)\le g$ is equivalent to $g^3-g^2-g-1\le 0$.  It is fulfilled since
$g\le 1.839$.  Thus $(\talpha,\tbeta)$ converges to $(\alpha,g\alpha)$ 
as $\epsilon\to 0$.  Since $\cS_{3,5}$ is a closed subset
of $[0,\infty]^2$, we conclude that it contains $(\alpha,g\alpha)$ and so,
by Theorem~\ref{dr:thm:rectangles}, it contains $(\alpha,\beta)$ for any
$\beta\in[g\alpha,\infty]$.  
\end{proof}

%
%

\section{Final remarks}
\label{sec:fin}

For a fixed $\uu\in\bR^n\setminus\{0\}$ and for each $q\in\bR$, we form the 
convex body of $\bR^n$
\[
 \tcC_\uu(q)=\{\ux\in\bR^n\,;\, \norm{\ux\wedge\uu}\le 1\ \text{and}\ |\ux\cdot\uu|\le e^{-q}\}
\]
where $\ux\cdot\uu$ stands for the usual scalar product of $\ux$ and $\uu$ in $\bR^n$.
Then, for each $i=1,\dots,n$, we denote by $\tL_i(q)$ the logarithm of the $i$-th
minimum of $\tcC_\uu(q)$ with respect to $\bZ^n$, namely the smallest $t\in\bR$
such that $e^t\tcC_\uu(q)$ contains $i$ linearly independent elements of $\bZ^n$.
Our first remark is that sign-free-$n$-systems on $\bR$ occur naturally as
approximations of the maps
\[
 \begin{array}{rcl}
  \tuL_\uu\colon\bR &\longrightarrow &\bR^n\\
  q &\longmapsto &(\tL_1(q),\dots,\tL_n(q)).
 \end{array}
\]

\begin{proposition}
\label{fin:prop1}
For each $\uu\in\bR^n\setminus\{0\}$, there exists a sign-free-$n$-system 
$\uP\colon\bR\to\bR^n$ with $\uP=\uP^\vee$ and $\uP(0)=0$ such 
that $\tuL_\uu-\uP$ is bounded. Conversely, for each sign-free-$n$-system 
$\uP\colon\bR\to\bR^n$ with $\uP=\uP^\vee$ and $\uP(0)=0$, there 
exists $\uu\in\bR^n\setminus\{0\}$ such that $\tuL_\uu-\uP$ is bounded.
\end{proposition}

\begin{proof}
This follows from the main result of \cite{R2015} together with Mahler's duality.  
We first note that $\tuL_\uu-\tuL_\uv$ is bounded on $\bR$ if $\uv=r\uu$ for 
some $r>0$.  So, we may assume that $\norm{\uu}=1$.  

For a fixed $\uu\in\bR^n$ with $\norm{\uu}=1$, consider the family of convex bodies 
\[
 \cC_\uu(q)=\{\ux\in\bR^n\,;\, \norm{\ux}\le 1\ \text{and}\ |\ux\cdot\uu|\le e^{-q}\}
\]
with parameter $q\in [0,\infty)$.  We define
\[
 \begin{array}{rcl}
  \uL_\uu\colon[0,\infty) &\longrightarrow &\bR^n\\
   q &\longmapsto &(L_1(q),\dots,L_n(q)),
 \end{array}
\]
where $L_i(q)$ is the logarithm of the $i$-th minimum of $\cC_\uu(q)$ 
with respect to $\bZ^n$ for $i=1,\dots,n$.
We note that each $\ux\in\bR^n$ decomposes uniquely as a sum
$\ux=\ux_\uu+(\ux\cdot\uu)\uu$ where $\ux_\uu$ is the orthogonal
projection of $\ux$ on $\uu^\perp$.  Then, we find $\norm{\ux\wedge\uu}=
\norm{\ux_\uu\wedge\uu}=\norm{\ux_\uu}$ and
\begin{equation}
\label{fin:prop1:eq1}
 \ux\cdot\uy = \ux_\uu\cdot\uy_\uu + (\ux\cdot\uu)(\uy\cdot\uu)
\end{equation}
for any $\ux,\uy\in\bR^n$.  In particular, we have the well-known formula
$\norm{\ux}^2=\norm{\ux\wedge\uu}^2+(\ux\cdot\uu)^2$ which implies that
\[
 \cC_\uu(q) \subseteq \tcC_\uu(q) \subseteq 2\cC_\uu(q)
\]
for each $q\ge 0$.  Thus, $\tuL_\uu-\uL_\uu$ is bounded on $[0,\infty)$.
Using \eqref{fin:prop1:eq1}, we also find, for each $q\in\bR$,  that the dual
(or polar) convex body $\tcC_\uu(q)^*$, defined as the set of all $\uy\in\bR^n$
such that $|\ux\cdot\uy| \le 1$ for each $\ux\in\tcC_\uu(q)$, satisfies
\[
 (1/2)\tcC_\uu(-q) \subseteq  \tcC_\uu(q)^*  \subseteq \tcC_\uu(-q).
\]
By Mahler's duality \cite[\S14, Theorem 5]{GL1987}, this implies that 
$\tL_i(-q)+\tL_{n+1-i}(q)$ is a bounded function of $q\in\bR$, for $i=1,\dots,n$.

For a given unit vector $\uu\in\bR^n$, Theorem~1.3 of \cite{R2015}
provides an $n$-system $\uP\colon[q_0,\infty)\to\bR^n$ such that 
$\uL_\uu-\uP$ is bounded on $[q_0,\infty)$.  We may assume that $q_0=0$
as such a map $\uP$ can be extended to an $n$-system on $[0,\infty)$
(see the proof of \cite[Theorem~8.1]{R2015}).  For $\epsilon>0$ small 
enough, we have $\uP(q)=(0,\dots,0,q)$ for each $q\in[0,\epsilon]$. So,
$\uP$ can be further extended to a sign-free-$n$-system on $\bR$ by
the formula $\uP(q)=\uP^\vee(q)$ for each $q\in (-\infty,0]$. Then, we 
have $\uP=\uP^\vee$ and $\uP(0)=0$.  By the above, $\tuL_\uu-\uP$ 
is bounded on $[0,\infty)$ and $\tuL_\uu-\uP^\vee$ is bounded 
on $(-\infty,0]$.   Thus, $\tuL_\uu-\uP$ is bounded on $\bR$.

Conversely, suppose that $\uP\colon\bR\to\bR^n$ is a sign-free-$n$-system 
with $\uP=\uP^\vee$ and $\uP(0)=0$.  Then, the restriction of $\uP$ to
$[0,\infty)$ is an $n$-system and so \cite[Theorem~8.1]{R2015} provides 
a unit vector $\uu\in\bR^n$ such that $\uL_\uu-\uP$ is bounded on
$[0,\infty)$.  We conclude as above that $\tuL_\uu-\uP$ is bounded 
on $\bR$.
\end{proof}

In a remark at the end of section \ref{sec:type}, we observed that, for any 
integers $m$ and $n$ with $1\le m\le n-1$, and for any non-degenerate 
proper $n$-system $\uP$, there is a potentially simpler map with the same 
$m$-division points providing the same point in $\cS_{n-m,n}$.  Our last remark
is that this process of deformation can be extended further for $\cS_{3,5}$ 
using the generalized $n$-systems from \cite{R2016}, as the next result shows.

\begin{proposition}
\label{fin:prop2}
Let $\uP=(P_1,\dots,P_5)$ be a non-degenerate
proper $5$-system.  Then, there exists a proper generalized $5$-system 
$\tuP=(\tP_1,\dots,\tP_5)$ with the following properties:
\begin{itemize}
\item[(i)] we have $\tuP=\uP$ on each interval $[a,b]$ in $\cI_2(\uP)$ 
on which $P_4$ or $P_5$ is not constant;
\smallskip
\item[(ii)]  on any other interval $[a,b]$ in $\cI_2(\uP)$, the ratio
\begin{equation}
 \label{fin:prop2:eq1}
 \tchi(q) = \frac{\tP_3(q)+\tP_4(q)+\tP_5(q)}{\tP_1(q)+\tP_2(q)}
\end{equation}
is (strictly) decreasing;
\smallskip
\item[(iii)] with the notation \eqref{pgn:eq:chi}, we have 
 $\chibot_2(\tuP)=\chibot_2(\uP)$ and   $\chitop_2(\tuP)\le \chitop_2(\uP)$.
\end{itemize}
\end{proposition}

Since $(\chibot_2(\tuP),\chitop_2(\tuP))\in \cS_{3,5}$, the above 
result could be useful to study the boundary of $\cS_{3,5}$.

\begin{proof}
Suppose first that $[a,b]$ is any interval in $\cI_2(\uP)$ of type 
$\{2,3\}$ or $\{1,2,3\}$, and consider the generalized $5$-system
$\tuP$ on $[a,b]$ whose combined graph is given in 
Figure~\ref{fin:fig3}, using the notation $a_i=P_i(a)$ and 
$b_i=P_i(b)$ for $i=1,\dots,5$.  In this figure, the numbers on 
the various slanted line segments indicate their slopes.

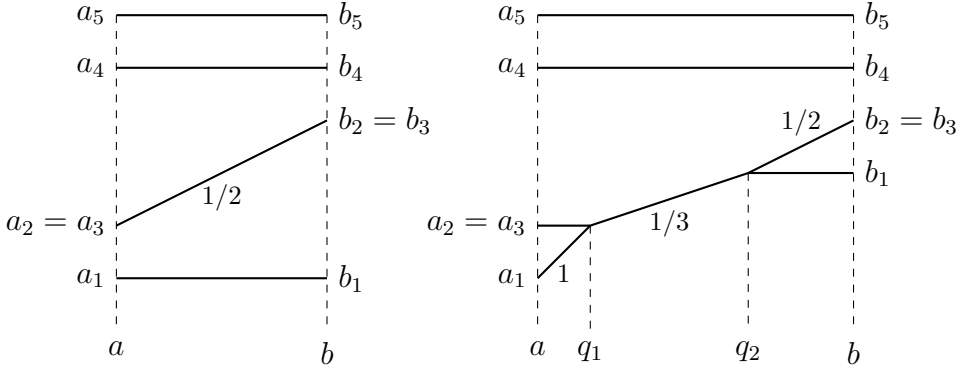
\begin{figure}[h]
     \begin{tikzpicture}[scale=0.7]
      \draw[dashed] (1,6) -- (1,0) node[below] {$a$};
      \draw[dashed] (5,6) -- (5,0) node[below] {$b$};
      \draw[thick] (1,1) -- (5,1) node[right] {$b_1$}; 
      \node[left] at (1,1){$a_1$};
      \draw[thick] (1,2) -- (5,4) node[right] {$b_2=b_3$}; 
      \node[left] at (1,2){$a_2=a_3$};
      \draw[thick] (1,5) -- (5,5) node[right] {$b_4$}; 
      \node[left] at (1,5){$a_4$};
      \draw[thick] (1,6) -- (5,6) node[right] {$b_5$}; 
      \node[left] at (1,6){$a_5$};
      \node[below] at (3,3){\footnotesize $1/2$};
      \draw[dashed] (9,6) -- (9,0) node[below] {$a$};
      \draw[dashed] (10,2) -- (10,0) node[below] {$q_1$};
      \draw[dashed] (13,3) -- (13,0) node[below] {$q_2$};
      \draw[dashed] (15,6) -- (15,0) node[below] {$b$};
      \draw[thick] (9,1) -- (10,2) -- (13,3) -- (15,3) node[right] {$b_1$}; 
      \node[left] at (9,1){$a_1$};
      \draw[thick] (10,2) -- (9,2) node[left] {$a_2=a_3$}; 
      \draw[thick]  (13,3) -- (15,4) node[right]{$b_2=b_3$};
      \draw[thick] (9,5) -- (15,5) node[right] {$b_4$}; 
      \node[left] at (9,5){$a_4$};
      \draw[thick] (9,6) -- (15,6) node[right] {$b_5$}; 
      \node[left] at (9,6){$a_5$};
      \node[below] at (9.5,1.45){\footnotesize $1$};
      \node[below] at (11.5,2.5){\footnotesize $1/3$};
      \node[above] at (14,3.5){\footnotesize $1/2$};
   \end{tikzpicture}
\caption{The two possibilities for the combined graph of $\tuP$ over $[a,b]$.}
\label{fin:fig3}
\end{figure}

We claim that the ratio $\tchi$ given by \eqref{fin:prop2:eq1} is 
decreasing on $[a,b]$ and that the ratio $\chi$ given by
\[
 \chi(q) =  \frac{P_3(q)+P_4(q)+P_5(q)}{P_1(q)+P_2(q)}
\]
satisfies
\[
 \min_{q\in[a,b]}\chi(q) = \chi(b) = \min_{q\in[a,b]}\tchi(q)
 \et
 \max_{q\in[a,b]}\chi(q) \ge \chi(a) = \max_{q\in[a,b]}\tchi(q).
\]
The second part of the claim follows from the first when $[a,b]$ is a simple 
interval because, as shown in the proof of Proposition~\ref{part:prop:n-sys}, 
the minimum of $\chi$ on $[a,b]$ is achieved at $a$ or $b$.  Since $\chi$ 
coincides with $\tchi$ at these points, we deduce from the first assertion that 
$\min_{q\in[a,b]}\chi(q) = \chi(b)$ if $[a,b]$ is simple.   Since $[a,b]$ partitions into 
simple intervals of type $\{2,3\}$ or $\{1,2,3\}$, this equality remains true in general.

To prove the first part of the claim, consider the sum $\tS:=\tP_1+\tP_2$.  
If $[a,b]$ has type $\{2,3\}$, its slope is $1/2$ on the whole interval $[a,b]$.  
If $[a,b]$ has type $\{1,2,3\}$, its slope is $1$ on $[a,q_1]$, then $2/3$ on 
$[q_1,q_2]$, and finally $1/2$ on $[q_2,b]$.  In both cases, the slope of 
$\tS$ is at least $1/2$ at each point of $[a,b]$ where it is differentiable.  
Moreover, for each $q\in[a,b]$, we have 
\[
 (3/2)\tS(q) \le 3\tP_2(q) 
                 \le \tP_3(q)+\tP_4(q)+\tP_5(q)= q - \tS(q),
\]
thus $\tS(q)/q \le 2/5 < 1/2$.  Consequently, the ratio $\tS(q)/q$ is 
increasing on $[a,b]$, and so $\tchi(q)=q/\tS(q)-1$ is 
decreasing on $[a,b]$.  

To finish the proof, it suffices to modify $\uP$ as in Figure~\ref{fin:fig3} 
on each maximal interval $[a,b]$ in $\cI_2(\uP)$ of type contained in
$\{1,2,3\}$.  The resulting map is a proper generalized $5$-system with 
the required properties (i)--(iii).  
\end{proof}



\subsection*{Acknowlegments.}  
This research was partially supported by an NSERC discovery grant.  
The authors thank Anthony Po\"els for useful comments on an earlier 
version of this paper.

\end{document}